\documentclass[a4paper]{article}
\title{On finiteness properties of the unramified Iwasawa module of a $\Z_p$-extension with restricted ramification}
\author{Takuya Yanagisawa\\
{\small Department of mathematics, Graduate School of Science, Tokyo University of Science}\\
{\small \texttt{1125705@ed.tus.ac.jp}}}
\date{}

\usepackage{amsmath, amssymb}
\usepackage{amsthm}
\usepackage{mathtools}
\usepackage{stmaryrd}
\usepackage{float}
\usepackage{url}
\usepackage[labelsep=none]{caption}
\usepackage{enumitem}

\usepackage[top=3cm,bottom=3cm,left=2cm,right=2cm]{geometry}

\theoremstyle{plain}
\newtheorem{df}{Definition}[section]
\newtheorem{thm}[df]{Theorem}
\newtheorem{prop}[df]{Proposition}
\newtheorem{cor}[df]{Corollary}
\newtheorem{lem}[df]{Lemma}
\newtheorem{condition}[df]{Condition}
\newtheorem{conjecture}[df]{Conjecture}
\newtheorem{prob}[df]{Problem}
\newtheorem*{convention}{Convention}

\theoremstyle{definition}
\newtheorem{rem}[df]{Remark}
\newtheorem{example}[df]{Example}

\usepackage{tikz}
\usetikzlibrary{arrows.meta}

\newcommand{\Z}{\mathbb{Z}}
\newcommand{\Q}{\mathbb{Q}}

\newcommand{\C}{\mathbb{C}}

\newcommand{\pe}{\mathfrak{p}}
\newcommand{\Pe}{\mathfrak{P}}
\newcommand{\qu}{\mathfrak{q}}
\newcommand{\Qu}{\mathfrak{Q}}

\newcommand{\ve}{\varepsilon}

\DeclareMathOperator{\Ker}{Ker}
\DeclareMathOperator{\Gal}{Gal}

\DeclareMathOperator{\rank}{rank}
\DeclareMathOperator{\Homo}{H}
\DeclareMathOperator{\Tor}{Tor}
\DeclareMathOperator{\Cl}{Cl}

\DeclareMathOperator{\Image}{Im}

\newcommand{\ab}{\mathrm{ab}}
\newcommand{\cy}{\mathrm{c}}
\newcommand{\an}{\mathrm{a}}
\newcommand{\fin}{\mathrm{fin}}

\newcommand{\tgen}[1]{\overline{\langle #1 \rangle}}
\newcommand{\br}[1]{\lbrack #1 \rbrack}
\newcommand{\bbr}[1]{\llbracket #1 \rrbracket}
\newcommand{\abs}[1]{\lvert #1 \rvert}

\usepackage[backend=biber, style=numeric]{biblatex}
\usepackage{comment}

\begin{document}
\maketitle

\begin{abstract}
In the present article, we consider $\Z_p$-extensions defined by restrictions on ramification at $p$-adic primes. We study the finiteness, the existence of non-trivial finite submodules, and the triviality of the unramified Iwasawa modules of these $\Z_p$-extensions. In addition, we give a new sufficient condition for Greenberg's Generalized Conjecture with some non-trivial concrete examples. 
\end{abstract}

\section{Introduction}
In classical Iwasawa theory, the Galois module structure of certain arithmetic objects is deeply studied. One of the most typical examples is the unramified Iwasawa module $X(K)$ of a multiple $\Z_p$-extension $K/k$, which has a natural $\Gal(K/k)$-module structure. The knowledge of the Galois module structure of $X(K)$ provides us rich arithmetic information of the intermediate fields of $K/k$. In the present article, we study some problems on whether $X(K)$ has a simple structure, in the case where $K/k$ is characterized by a restriction on ramifications. 

Here we introduce some notation. Let $k/\Q$ be an algebraic extension and $p$ a prime number. We let $S_p(k)$ denote the set of all $p$-adic primes of $k$. For a subset $\Sigma$ of $S_p(k)$, we let $M_\Sigma(k)/k$ denote the maximal abelian pro-$p$-extension that is unramified at all primes not contained in $\Sigma$, and $X_\Sigma(k)\coloneqq\Gal(M_\Sigma(k)/k)$ its Galois group. Let $K/k$ be an algebraic extension and put $S_\Sigma(K)\coloneqq\{\Pe\in S_p(K)\mid\Pe\cap k\in\Sigma\}$. We use the notation $M_\Sigma(K)\coloneqq M_{S_\Sigma(K)}(K)$ and $X_\Sigma(K)\coloneqq X_{S_\Sigma(K)}(K)$ for simplicity. For $\pe\in S_p(k)$, we also use the notation $S_\pe(K)\coloneqq S_{\{\pe\}}(K)$, $M_\pe(K)\coloneqq M_{\{\pe\}}(K)$, and $X_\pe(K)\coloneqq X_{\{\pe\}}(K)$. In the case where $\Sigma=\emptyset$, we put $L(K)\coloneqq M_\emptyset(K)$ and $X(K)\coloneqq X_\emptyset(K)$. When $K/k$ is a Galois extension, we regard $X_\Sigma(K)$ as a module over the completed group algebra $\Z_p\bbr{\Gal(K/k)}$. In the case where $K/k$ is a $\Z_p^d$-extension for some $d\geq 1$, the $\Z_p\bbr{\Gal(K/k)}$-module $X_\Sigma(K)$ is called the $\Sigma$-ramified Iwasawa module of $K/k$. In particular, if $\Sigma=\emptyset$, the module $X(K)$ is called the unramified Iwasawa module of $K/k$.

\subsection{Greenberg's Conjecture}
For a finite extension $k/\Q$ and a prime number $p$, we let $k^\cy_\infty/k$ denote the cyclotomic $\Z_p$-extension over $k$. In the present article, we consider certain variants of the following conjecture. 
\begin{conjecture}[Greenberg's Conjecture \cite{Greenberg1976}]\label{conjecture Greenberg}
Let $k$ be a totally real number field and $p$ a prime number. Then $X(k^\cy_\infty)$ is finite. 
\end{conjecture}
Conjecture~\ref{conjecture Greenberg} has been investigated through many examples by numerous authors, but only few theoretical results are known so far.

Instead of the cyclotomic $\Z_p$-extension over a totally real number field, we consider a $\Z_p$-extension defined by a restriction on ramifications. Let $k/\Q$ be a finite extension, $p$ a prime number, and $\Sigma$ a set of $p$-adic primes of $k$. A $\Z_p$-extension $N_\infty/k$ is called a $\Sigma$-ramified $\Z_p$-extension if $N_\infty$ is contained in $M_\Sigma(k)$. If the $\Z_p$-rank of $X_\Sigma(k)$ is $1$, then we let $N_\infty/k$ denote the unique $\Sigma$-ramified $\Z_p$-extension unless otherwise stated. 
Concerning Conjecture~\ref{conjecture Greenberg}, the following problem has been considered (for example, see \cite[Problem~1.1]{KM2024}). 
\begin{prob}[an analogue of Greenberg's Conjecture]\label{problem Greenberg}
Let $k/\Q$ be a finite extension, $p$ a prime number, and $\Sigma$ a set of $p$-adic primes of $k$. Suppose that there exists a unique $\Sigma$-ramified $\Z_p$-extension $N_\infty/k$. Then determine whether $X(N_\infty)$ is finite or not. 
\end{prob}
If $k$ is a totally real number field and $p$-adic Leopoldt's Conjecture is true for $k$, then the cyclotomic $\Z_p$-extension $k^\cy_\infty/k$ is the unique $\Z_p$-extension over $k$ and, in particular, it is the unique $S_p(k)$-ramified $\Z_p$-extension. Thus Conjecture~\ref{conjecture Greenberg} is included in Problem~\ref{problem Greenberg}. 

Although negative examples of Problem~\ref{problem Greenberg} are known (see \cite[Remark~1.4]{KM2024} and its reference), it is still an interesting problem. In particular, in the case where $k$ is a CM-field and $\Sigma$ is a $p$-adic CM-type (see \S\ref{subsection CM-fields} for the definition), it seems that no negative example of Problem~\ref{problem Greenberg} is known so far (see \cite[Conjecture~1.2]{PS2024}). 

Here we note some advantages of considering such type of $\Z_p$-extensions. 
\begin{itemize}
\item We will observe that many known results on the cyclotomic $\Z_p$-extensions of totally real number fields can be extended to $\Sigma$-ramified setting. So we can enlarge the applicability of earlier works on totally real number fields. 

\item While only few theoretical results on Conjecture~\ref{conjecture Greenberg} are known, one may find out the essence of Conjecture~\ref{conjecture Greenberg} (the reason why $k$ has to be totally real) by considering Problem~\ref{problem Greenberg}. 

\item The study of $\Sigma$-ramified $\Z_p$-extensions can be applied to that of Greenberg's Generalized Conjecture (Conjecture~\ref{conjecture GGC}), as some authors have done and we will do in Theorem~\ref{thm main GGC}. Greenberg's Generalized Conjecture is a conjecture on the maximal multiple $\Z_p$-extension $\widetilde{k}/k$, which is a canonical object in the sense of the independence from the choice of $\Sigma$. Moreover, we can derive information of other $\Z_p$-extensions, including the cyclotomic $\Z_p$-extension, from the conjecture (see \S\ref{subsection applications}). So the study of $\Sigma$-ramified $\Z_p$-extensions goes beyond analogies with the cyclotomic $\Z_p$-extension. 
\end{itemize}

With regard to Problem~\ref{problem Greenberg}, we provide a generalization of Greenberg's theorem concerning Conjecture~\ref{conjecture Greenberg} to $\Sigma$-ramified setting. 

Here we set up some notation. Let $p$ be a prime number. For any algebraic extension $F/\Q$, we let $\Cl(F)$ denote the ideal class group of $F$ and put $A(F)\coloneqq\Z_p\otimes_\Z \Cl(F)$. Here we note that $\Cl(F)$ can be regarded as the injective limit $\varinjlim_{i}\Cl(F_i)$, where $(F_i)_i$ is a family of intermediate finite extensions of $F/\Q$ such that $\bigcup_i F_i=F$ and the limit is taken with respect to the lifting maps. Then we also have $A(F)=\varinjlim_{i}A(F_i)$. Let $k/\Q$ be a finite extension and $k_\infty/k$ a $\Z_p$-extension. For each $n\geq 0$, we let $k_n$ denote the unique intermediate field of $k_\infty/k$ such that $[k_n:k]=p^n$. For $0\leq n\leq m\leq\infty$, let $i_{n, m}\colon A(k_n)\to A(k_m)$ be the lifting map. 

With the above notation, we show the following. 
\begin{thm}\label{thm main GC}
Let $k/\Q$ be a finite extension, $p$ a prime number, and $\Sigma$ a set of $p$-adic primes of $k$. Let $N_\infty/k$ be a $\Sigma$-ramified $\Z_p$-extension that is totally ramified at every ramified prime. Suppose also that one of the following two conditions holds. 
\begin{enumerate}[label=(\Alph*)]
\item $\Sigma$ consists of a single prime. 

\item $U_\Sigma^{(1)}/(\Z_p\otimes_\Z E(k))$ is a free $\Z_p$-module of rank $1$ (see \S\ref{section Leopoldt} for the notation). 
\end{enumerate}
Then the following are equivalent. 
\begin{itemize}
\item[(i)] $X(N_\infty)$ is finite. 

\item[(ii)] The lifting map $i_{0, \infty}\colon A(k)\to A(N_\infty)$ is trivial. 
\end{itemize}
\end{thm}
Case~\ref{A} of this theorem recovers Greenberg's result \cite[Theorem~1]{Greenberg1976} when applied to a totally real number field $k$ and $\Sigma=S_p(k)$. One can prove Case~\ref{A} of Theorem~\ref{thm main GC} by mimicking Greenberg's proof of \cite[Theorem~1]{Greenberg1976}. On the other hand, Fujii \cite{Fujii2020} gave another proof of Greenberg's result. One may also show Case~\ref{A} of Theorem~\ref{thm main GC} by Fujii's method. Our proof of Theorem~\ref{thm main GC} is based on Fujii's method, but we need a slight modification to deal with Case~\ref{B}.

\subsection{A weak form of Greenberg's Conjecture}
Next we consider a weak form of Conjecture~\ref{conjecture Greenberg}. 
\begin{conjecture}[a weak form of Greenberg's Conjecture]\label{conjecture weak form}
Let $k$ be a totally real number field and $p$ a prime number. If $X(k^\cy_\infty)$ is not trivial, then $X(k^\cy_\infty)$ has a non-trivial finite $\Z_p\bbr{\Gal(k^\cy_\infty/k)}$-submodule. 
\end{conjecture}
Conjecture~\ref{conjecture weak form} was proposed by Nguyen Quang Do \cite[(GF)]{Nguyen2006} (the above formulation can be found in \cite[(WGC)]{Nguyen2017}). 

One may also consider a weak form of Problem~\ref{problem Greenberg}. 
\begin{prob}[a weak form of an analogue of Greenberg's Conjecture]\label{problem weak form}
Let $k/\Q$ be a finite extension, $p$ a prime number, and $\Sigma$ a set of $p$-adic primes of $k$. Suppose that there exists a unique $\Sigma$-ramified $\Z_p$-extension $N_\infty/k$. If $X(N_\infty)$ is not trivial, then determine whether $X(N_\infty)$ has a non-trivial finite $\Z_p\bbr{\Gal(N_\infty/k)}$-submodule or not. 
\end{prob}
With regard to Problem~\ref{problem weak form}, we provide generalizations of some theorems concerning Conjecture~\ref{conjecture weak form} to $\Sigma$-ramified setting. 

The following is a variant of Theorem~\ref{thm main GC}. 
\begin{thm}\label{thm WGC 1}
Let $k/\Q$ be a finite extension, $p$ a prime number, and $\Sigma$ a set of $p$-adic primes of $k$. Let $N_\infty/k$ be a $\Sigma$-ramified $\Z_p$-extension that is totally ramified at every ramified prime. Suppose also that one of the following two conditions holds. 
\begin{enumerate}[label=(\Alph*)]
\item $\Sigma$ consists of a single prime. 

\item $U_\Sigma^{(1)}/(\Z_p\otimes_\Z E(k))$ is a free $\Z_p$-module of rank $1$. 
\end{enumerate}
Then the following are equivalent. 
\begin{itemize}
\item[(i)] $X(N_\infty)$ has no non-trivial finite $\Z_p\bbr{\Gal(N_\infty/k)}$-submodule. 

\item[(ii)] The lifting map $i_{0, \infty}\colon A(k)\to A(N_\infty)$ is injective. 
\end{itemize}
\end{thm}
More precisely, we show that the size of the kernel of the lifting map corresponds to the size of the maximal finite $\Z_p\bbr{\Gal(N_\infty/k)}$-submodule of $X(N_\infty)$. See Theorem~\ref{thm A and B} and Remark~\ref{remark kernel finite} for the precise statement. 

Case~\ref{A} of Theorem~\ref{thm WGC 1} recovers Fujii's result \cite[Theorem~1]{Fujii2020} when applied to a totally real number field $k$ and $\Sigma=S_p(k)$. Our proof of Theorem~\ref{thm WGC 1} is based on Fujii's method and proceeds in parallel with that of Theorem~\ref{thm main GC}. 

In the following theorems, we assume $\Sigma$-adic Leopoldt's Conjecture instead of usual $p$-adic Leopoldt's Conjecture. While we introduce the concept of $\Sigma$-adic Leopoldt's Conjecture in \S\ref{section Leopoldt}, we note here that $S_p(k)$-adic Leopoldt's Conjecture coincides with usual $p$-adic Leopoldt's Conjecture. 

From now on, we use the following convention. 
\begin{convention}
Let $N_\infty/k$ be a $\Sigma$-ramified $\Z_p$-extension. We say that $N_\infty/k$ is ramified (resp. totally ramified) at $\Sigma$ if $N_\infty/k$ is ramified (resp. totally ramified) at every prime in $\Sigma$. 
\end{convention}
\begin{thm}\label{thm main WGC 2}
Let $k/\Q$ be a finite extension, $p$ an odd prime number, and $\Sigma$ a set of $p$-adic primes of $k$ that are of degree $1$. Suppose that the $\Z_p$-rank of $X_\Sigma(k)$ is $1$ and the unique $\Sigma$-ramified $\Z_p$-extension $N_\infty/k$ is ramified at $\Sigma$. Suppose also that $\Sigma$-adic Leopoldt's Conjecture holds for $k$. Then the following are equivalent. 
\begin{itemize}
\item[(i)] $X(N_\infty)$ has no non-trivial finite $\Z_p\bbr{\Gal(N_\infty/k)}$-submodule. 

\item[(ii)] $L(N_\infty)=M_\Sigma(N_\infty)$ holds. 
\end{itemize}
\end{thm}
This recovers Ozaki's result obtained in the proof of \cite[Theorem~2]{Ozaki1997} when applied to a totally real number field $k$ and $\Sigma=S_p(k)$. Our proof of Theorem~\ref{thm main WGC 2} is a direct generalization of that of Ozaki. 

We show a similar result under different assumptions. 
\begin{thm}\label{thm main WGC 3}
Let $k/\Q$ be a finite extension, $p$ a prime number, and $\Sigma$ a set of $p$-adic primes of $k$. Suppose that the $\Z_p$-rank of $X_\Sigma(k)$ is $1$ and let $N_\infty/k$ be the unique $\Sigma$-ramified $\Z_p$-extension. Suppose also that $\Sigma$-adic Leopoldt's Conjecture holds for $k$ and $U_\Sigma^{(1)}/(\Z_p\otimes_\Z E(k))$ is free as a $\Z_p$-module. Then the following are equivalent. 
\begin{itemize}
\item[(i)] $X(N_\infty)$ has no non-trivial finite $\Z_p\bbr{\Gal(N_\infty/k)}$-submodule. 

\item[(ii)] $L(N_\infty)=M_\Sigma(N_\infty)$ holds. 
\end{itemize}
\end{thm}
This recovers Fujii's result \cite[Theorem~2]{Fujii2020} when applied to a totally real number field $k$, and our proof is based on Fujii's method (with a slight modification). We note that Theorem~\ref{thm main WGC 3} implies a stronger result compared to \cite[Theorem~2]{Fujii2020}, even when applied to a totally real number field. 
\begin{cor}
Let $k$ be a totally real number field and $p$ a prime number. Suppose that $p$-adic Leopoldt's Conjecture is true for $k$ and $U_{S_p(k)}^{(1)}/(\Z_p\otimes_\Z E(k))$ is free as a $\Z_p$-module. Then the following are equivalent. 
\begin{itemize}
\item[(i)] $X(k^\cy_\infty)$ has no non-trivial finite $\Z_p\bbr{\Gal(k^\cy_\infty/k)}$-submodule. 

\item[(ii)] $L(k^\cy_\infty)=M_p(k^\cy_\infty)$ holds. 
\end{itemize}
\end{cor}
This corollary completely dispenses with the assumptions on the ramification in \cite[Theorem~2]{Fujii2020}. However, in the case where $p=2$, the assumption on the freeness of $U_{S_2(k)}^{(1)}/(\Z_2\otimes_{\Z} E(k))$ is satisfied only when $2$ does not split in $k$.

\subsection{The triviality of the unramified Iwasawa module}
Next we consider the triviality of $X(N_\infty)$. For a $\Z_p$-extension $k_\infty/k$, the triviality of $X(k_\infty)$ is equivalent to the vanishing of all the Iwasawa invariants of $k_\infty/k$. 

The following gives a criterion on the triviality of $X(N_\infty)$. 
\begin{thm}\label{thm triviality}
Let $k/\Q$ be a finite extension, $p$ an odd prime number, and $\Sigma$ a set of $p$-adic primes of $k$ that are of degree $1$. Let $N_\infty/k$ be a $\Sigma$-ramified $\Z_p$-extension. Suppose that $N_\infty/k$ is ramified at $\Sigma$ and totally ramified at some prime in $\Sigma$. Then the following are equivalent. 
\begin{enumerate}[label=(\roman*)]
\item $X(N_\infty)$ is trivial. \label{one}

\item $p$ does not divide the class number of $k$ and $U_\Sigma^{(1)}/(\Z_p\otimes_{\Z}E(k))$ is a free $\Z_p$-module of rank $1$. \label{two}
\end{enumerate}
\end{thm}
This recovers a part of Fukuda--Komatsu's result \cite[Theorem~1, Corollary]{FuKo1986} when applied to a real quadratic field and $\Sigma=S_p(k)$ (see Corollary~\ref{cor FK} as well). 

Theorem~\ref{thm triviality} also recovers Jaulent's result \cite[Theorem~6]{Jaulent2024} when applied to $\Sigma=S_p(k)$ and $N_\infty=k^\cy_\infty$. In this setting, Jaulent showed that the triviality of $X(k^\cy_\infty)$ is equivalent to the assertion that $k$ is totally real and $p$-rational (see \cite[Definition~5]{Jaulent2024}), which is equivalent to \ref{two}. Our proof of Theorem~\ref{thm triviality} is similar to that of \cite[Theorem~6]{Jaulent2024}. 
\begin{rem}
It is noteworthy that Hachimori \cite{Hachimori2003} dealt with $\Sigma$-ramified Iwasawa modules of the cyclotomic $\Z_p$-extension over number fields and generalized Fukuda--Komatsu's result in another direction (\cite[Theorem~8.1]{Hachimori2003}). Hachimori's theorem is similar to ours but our theorem does not apply to the setting in \cite{Hachimori2003}. 
\end{rem}
\begin{rem}
Some authors have obtained parts of the results stated above in some special settings as an intermediate step of the proof of Greenberg's Generalized Conjecture (Conjecture~\ref{conjecture GGC}). For example, Itoh \cite{Itoh2011} dealt with imaginary abelian quartic fields, Fujii \cite{Fujii2017} dealt with general CM-fields, and Kataoka \cite{Kataoka2017} dealt with complex cubic fields. 
\end{rem}

\subsection{Greenberg's Generalized Conjecture}\label{subsection intro GGC}
Next we consider a generalized form of Greenberg's Conjecture. For a finite extension $k/\Q$ and a prime number $p$, we let $\widetilde{k}/k$ denote the maximal multiple $\Z_p$-extension over $k$, that is, the composite of all $\Z_p$-extensions over $k$. 
\begin{conjecture}[Greenberg's Generalized Conjecture {\cite[Conjecture~(3.5)]{Greenberg1998}}]\label{conjecture GGC}
Let $k/\Q$ be a finite extension and $p$ a prime number. Then $X(\widetilde{k})$ is pseudo-null as a $\Z_p\bbr{\Gal(\widetilde{k}/k)}$-module. 
\end{conjecture}
Here, for a $\Z_p^d$-extension $K/k$ for some $d\geq 1$, a finitely generated $\Z_p\bbr{\Gal(K/k)}$-module $M$ is said to be pseudo-null if the height of the annihilator ideal of $M$ is greater than $1$. In the case where $d=1$, a finitely generated $\Z_p\bbr{\Gal(K/k)}$-module $M$ is pseudo-null if and only if it is finite. So Conjecture~\ref{conjecture GGC} is actually a generalization of Conjecture~\ref{conjecture Greenberg}, under $p$-adic Leopoldt's Conjecture. 

The following theorem of Fujii is one of the most notable results toward Conjecture~\ref{conjecture GGC}. 
\begin{thm}[{\cite[Theorem~2]{Fujii2017}}]\label{thm Fujii GGC}
Let $k$ be a CM-field and $p$ an odd prime number that splits completely in $k$. Suppose that $p$ does not divide the class number of $k$ and all the Iwasawa invariants of the cyclotomic $\Z_p$-extension of $k^+$ are trivial, where $k^+$ is the maximal totally real subfield of $k$. Then Conjecture~\ref{conjecture GGC} is true for $k$ and $p$. 
\end{thm}
This is a generalization of Itoh's result \cite[Theorem~1.1]{Itoh2011} on imaginary abelian quartic fields. We also mention that the case where $k$ is an imaginary quadratic field is a result of Minardi \cite[Proposition~3.A]{Minardi1986}. Here we remark that the assumption on $p$-adic Leopoldt's Conjecture is not needed in \cite[Theorem~2]{Fujii2017}, as mentioned in \cite[Remark~3.4~(2)]{Fujii2022}. 

On the other hand, Kataoka dealt with complex cubic fields. 
\begin{thm}[{\cite[Theorem~1.3]{Kataoka2017}}]\label{thm Kataoka}
Let $k/\Q$ be a cubic extension that has a complex place and $p$ an odd prime number that splits completely in $k$. Suppose that $p$ splits finitely in $\widetilde{k}$. Suppose also that $p$ does not divide the class number of $k$ and $p^2$ does not divide $\varepsilon^{p-1}$, where $\varepsilon$ is a fundamental unit of $k$. Then Conjecture~\ref{conjecture GGC} is true for $k$ and $p$. 
\end{thm}
We note that Kataoka showed Conjecture~\ref{conjecture GGC} in a more general setting (see \cite[Theorem~3.3]{Kataoka2017}). 

In the present article, we show the following by our observation on $\Sigma$-adic Leopoldt's Conjecture. 
\begin{thm}\label{thm main GGC}
Let $k/\Q$ be a finite extension, $p$ an odd prime number that splits completely in $k$, and $\Sigma$ a set of $p$-adic primes of $k$. Suppose that the following hold. 
\begin{enumerate}[label=(\Alph*)]
\item $p$ does not divide the class number of $k$. \label{GGC A}

\item $\Sigma$-adic Leopoldt's Conjecture holds for $k$. \label{GGC B}

\item The $\Z_p$-rank of $X_\Sigma(k)$ is $1$.  \label{GGC C}

\item $U_{\Sigma\setminus\{\pe\}}^{(1)}/(\Z_p\otimes_\Z E(k))$ is trivial for every $\pe\in\Sigma$. \label{GGC D}

\item Every $\qu\in S_p(k)\setminus\Sigma$ splits finitely in $M_\Sigma(k)/k$. \label{GGC E}

\item $\qu'$ splits finitely in $M_{(\Sigma\setminus\{\pe\})\cup\{\qu\}}(k)/k$, for every $\pe\in\Sigma$ and $\qu, \qu'\in S_p(k)\setminus\Sigma$ with $\qu\neq \qu'$. \label{GGC F}
\end{enumerate}
Then Conjecture~\ref{conjecture GGC} is true for $k$ and $p$. 
\end{thm}
Our proof of Theorem~\ref{thm main GGC} is a generalization of Fujii's proof of Theorem~\ref{thm Fujii GGC}, which is based on Itoh's proof of \cite[Theorem~1.1]{Itoh2011}. We construct a certain tower 
\[
k\subseteq K^{(1)}\subseteq K^{(2)}\subseteq \cdots\subseteq K^{(r_2(k)+1)}=\widetilde{k}
\]
of multiple $\Z_p$-extensions over $k$ and inductively show the pseudo-nullity of the Iwasawa modules of the intermediate extensions. 
\begin{rem}\label{rem conditions}
The mathematical arguments underlying this remark were developed with the assistance of ChatGPT. If an odd prime number $p$ splits completely in a finite extension $k/\Q$, then there exists a set $\Sigma$ of $p$-adic primes of $k$ that satisfies Conditions~\ref{GGC B} and \ref{GGC C}, under usual $p$-adic Leopoldt's Conjecture (Proposition~\ref{prop existence}). Moreover, if we assume weak $p$-adic Schanuel Conjecture, then there exists a $\Sigma$ that satisfies Conditions~\ref{GGC B}, \ref{GGC C}, \ref{GGC E}, \ref{GGC F}, and a certain weak version of \ref{GGC D} (Proposition~\ref{prop existence of Sigma} and Lemma~\ref{lem BCD'}). See Theorem~\ref{thm conditional GGC} for a conditional version of Theorem~\ref{thm main GGC}. These observations suggest that our basic strategy for Theorem~\ref{thm main GGC} may be developed in the future to resolve Conjecture~\ref{conjecture GGC} for arbitrary number fields in which $p$ splits completely (without specifying $\Sigma$). 

Consequently, corresponding Conditions~\ref{d} of Corollary~\ref{cor GGC abel imag quad} and \ref{GGC d} of Corollary~\ref{cor GGC single complex} are conjecturally automatic. 
\end{rem}

Theorem~\ref{thm main GGC} includes Theorems~\ref{thm Fujii GGC} and \ref{thm Kataoka} as special cases, as we will explain in \S\ref{subsection earlier}. In addition, we can deal with the following cases. 
\begin{cor}\label{cor GGC abel imag quad}
Let $K$ be an imaginary quadratic field, $k/K$ a finite abelian extension, and $p$ an odd prime number that splits completely in $k$. Let $\Pe$ be a $p$-adic prime of $K$ and put $\Sigma\coloneqq S_\Pe(k)$. Suppose that the following hold. 
\begin{enumerate}[label=(\alph*)]
\item $p$ does not divide the class number of $k$. \label{a}

\item $U_\Sigma^{(1)}/(\Z_p\otimes_{\Z}E(k))$ is free as a $\Z_p$-module. \label{b}

\item $L(K)\cap\widetilde{K}=K$ holds. \label{c}

\item $\qu'$ splits finitely in $M_{(\Sigma\setminus\{\pe\})\cup\{\qu\}}(k)/k$, for every $\pe\in\Sigma$ and $\qu, \qu'\in S_p(k)\setminus\Sigma$ with $\qu\neq \qu'$. \label{d}
\end{enumerate}
Then Conjecture~\ref{conjecture GGC} is true for $k$ and $p$. 
\end{cor}
\begin{cor}\label{cor GGC single complex}
Let $k/\Q$ be a finite extension that has a single complex place, $p$ an odd prime number that splits completely in $k$, and $\pe_0$ a $p$-adic prime of $k$. Put $\Sigma\coloneqq S_p(k)\setminus\{\pe_0\}$. Suppose that the following hold. 
\begin{enumerate}[label=(\alph*)]
\item $p$ does not divide the class number of $k$. \label{GGC a}


\item $U_{\Sigma\setminus\{\pe\}}^{(1)}/(\Z_p\otimes_{\Z}E(k))$ is trivial for every $\pe\in\Sigma$. \label{GGC c}

\item $\pe_0$ splits finitely in $M_\Sigma(k)/k$. \label{GGC d} 
\end{enumerate}
Then Conjecture~\ref{conjecture GGC} is true for $k$ and $p$. 
\end{cor}
By using the above two corollaries, we give some concrete examples of Conjecture~\ref{conjecture GGC} (see \S\ref{subsection example GGC}). 

There is an application of Theorem~\ref{thm main GGC} to non-abelian Iwasawa theory (see \S\ref{subsubsection non-abelian}) and an application of Corollary~\ref{cor GGC abel imag quad} to Iwasawa invariants of $\Z_p$-extensions (see \S\ref{subsubsection Iwasawa invariants}).

\subsection*{Structure of the present article}
The structure of the present article is as follows. In \S\ref{section notation}, we set up some notation and provide some basic properties which will be frequently used throughout the present article. In \S\ref{section Leopoldt}, we introduce the concept of $\Sigma$-adic Leopoldt's Conjecture. In \S\ref{section Sigma ramified}, we gather basic properties of $\Sigma$-ramified $\Z_p$-extensions. We prove Theorems~\ref{thm main GC} and \ref{thm WGC 1} in \S\ref{section 5} and prove Theorems~\ref{thm main WGC 2}, \ref{thm main WGC 3}, and \ref{thm triviality} in \S\ref{section 6}. In \S\ref{section Examples}, we show some examples of $\Sigma$-adic Leopoldt's Conjecture. In \S\ref{section GGC}, we prove Theorem~\ref{thm main GGC} and its corollaries. We also provide some applications (\S\ref{subsection applications}) and concrete examples (\S\ref{subsection example GGC}) of Theorem~\ref{thm main GGC} in \S\ref{section GGC}.

\section{Fundamental notation and properties}\label{section notation}
Here we gather some basic notation and properties that will be frequently used throughout the present article. 

For a given profinite group $\Gamma$ and a profinite (or discrete) $\Gamma$-module $M$, we let $M^\Gamma$ (resp. $M_\Gamma$) denote the $\Gamma$-invariant submodule (resp. $\Gamma$-coinvariant quotient) of $M$. For example, if $\Gamma$ is isomorphic to $\Z_p$ as a pro-$p$-group and $M$ is a pro-$p$-group (or a $p$-primary discrete $\Gamma$-module) for some prime number $p$, then $M$ has a natural $\Z_p\bbr{\Gamma}$-module structure and, $M^\Gamma=M\br{\gamma-1}$ and $M_\Gamma=M/(\gamma-1)M$ hold, where $\gamma\in \Gamma$ is any topological generator of $\Gamma$. Here $M\br{\gamma-1}$ denotes the kernel of the map $M\to M$ defined by multiplication by $\gamma-1$. Thus, in this case, these modules fit into an exact sequence
\[
\begin{tikzpicture}
\node (a) at (0,0) {$0$};
\node (b) at (2,0) {$M^\Gamma$}; 
\node (c) at (4,0) {$M$}; 
\node (d) at (6,0) {$M$}; 
\node (e) at (8,0) {$M_\Gamma$}; 
\node (f) at (10,0) {$0$}; 
\draw[->] (a) -- (b); \draw[->] (b) -- (c); \draw[->] (c) -- node[auto]{$\scriptstyle \gamma-1$} (d); \draw[->] (d) -- (e); \draw[->] (e) -- (f); 
\end{tikzpicture}
\]
of $\Z_p\bbr{\Gamma}$-modules. In this situation, for an exact sequence
\[
\begin{tikzpicture}
\node (a) at (0,0) {$0$};
\node (b) at (2,0) {$A$}; 
\node (c) at (4,0) {$B$}; 
\node (d) at (6,0) {$C$}; 
\node (e) at (8,0) {$0$};  
\draw[->] (a) -- (b); \draw[->] (b) -- (c); \draw[->] (c) -- (d); \draw[->] (d) -- (e); 
\end{tikzpicture}
\]
of $\Z_p\bbr{\Gamma}$-modules, we have an exact sequence
\[
\begin{tikzpicture}
\node (a) at (0,0) {$0$};
\node (b) at (2,0) {$A^\Gamma$}; 
\node (c) at (4,0) {$B^\Gamma$}; 
\node (d) at (6,0) {$C^\Gamma$}; 
\node (e) at (8,0) {$A_\Gamma$}; 
\node (f) at (10,0) {$B_\Gamma$}; 
\node (g) at (12,0) {$C_\Gamma$}; 
\node (h) at (14,0) {$0$}; 
\draw[->] (a) -- (b); \draw[->] (b) -- (c); \draw[->] (c) -- (d); \draw[->] (d) -- (e); \draw[->] (e) -- (f); \draw[->] (f) -- (g); \draw[->] (g) -- (h); 
\end{tikzpicture}
\]
of $\Z_p\bbr{\Gamma}$-modules by Snake Lemma. Hence, for a given $p$-primary discrete $\Gamma$-module $M$, there is an isomorphism $\Homo^1(\Gamma, M)\simeq M_\Gamma$ of $\Z_p$-modules, which depends on the choice of $\gamma$,  and $H^i(\Gamma, M)$ is trivial for every $i\geq 2$. 

The following lemma is well-known (for example, it is essentially shown in \cite[Lemma~1]{Ozaki2001}). 
\begin{lem}\label{lem five-term}
Let $p$ be a prime number, $k/\Q$ an algebraic extension, $F/k$ a pro-$p$-extension, and $K/F$ a $\Z_p$-extension such that $K/k$ is a Galois extension. Let $\Sigma$ be a set of $p$-adic primes of $k$. Then there is an exact sequence
\[
\begin{tikzpicture}
\node (a) at (0,0) {$0$}; 
\node (b) at (2.5,0) {$X_\Sigma(K)_{\Gal(K/F)}$}; 
\node (c) at (6.7,0) {$\Gal(M_\Sigma(K)\cap F^\ab/F)$}; 
\node (d) at (10.5,0) {$\Gal(K/F)$}; 
\node (e) at (12.4,0) {$0$};
\draw[->] (a) -- (b); \draw[->] (b) -- (c); \draw[->] (c) -- (d); \draw[->] (d) -- (e); 
\end{tikzpicture}
\]
of $\Z_p\bbr{\Gal(F/k)}$-modules, where $F^\ab/F$ denotes the maximal abelian extension. 
\end{lem}
\begin{proof}
Put $\Gamma\coloneqq\Gal(K/F)$. Then there is an exact sequence
\[
\begin{tikzpicture}
\node (a) at (0,0) {$\Homo_2(\Gamma, \Z_p)$}; 
\node (b) at (3.5,0) {$\Homo_1(X_\Sigma(K), \Z_p)_\Gamma$}; 
\node (c) at (8,0) {$\Homo_1(\Gal(M_\Sigma(K)/F), \Z_p)$}; 
\node (d) at (12.2,0) {$\Homo_1(\Gamma, \Z_p)$}; 
\node (e) at (14.6,0) {$0$};
\draw[->] (a) -- (b); \draw[->] (b) -- (c); \draw[->] (c) -- (d); \draw[->] (d) -- (e); 
\end{tikzpicture}
\]
of $\Z_p\bbr{\Gal(F/k)}$-modules, which is a part of the five-term exact sequence in group homology (see \cite[Corollary~7.2.6]{RZ2010}, for example). Since $\Gamma$ is cyclic as a $\Z_p$-module, we have $\Homo_2(\Gamma, \Z_p)=0$. One can see $\Homo_1(\Gamma, \Z_p)\simeq\Gamma^\ab=\Gamma$ by general group theory (see \cite[Lemma~6.8.6~(b)]{RZ2010}). We also have $\Homo_1(\Gal(M_\Sigma(K)/F), \Z_p)\simeq\Gal(M_\Sigma(K)/F)^\ab=\Gal(M_\Sigma(K)\cap F^\ab/F)$. We complete the proof by combining the above. 
\end{proof}

\section{On $\Sigma$-adic Leopoldt's Conjecture}\label{section Leopoldt}
In this section, we introduce the concept of $\Sigma$-adic Leopoldt's Conjecture. Here we note that such concept is not new and has been studied, for example, by Maire in \cite{Maire2002}. 

Let $k/\Q$ be a finite extension and $p$ a prime number. Let $E(k)$ be the unit group of $k$. For each $\pe\in S_p(k)$, let $k_\pe$ be the completion of $k$, $U_\pe$ the local unit group, and $U_\pe^{(1)}$ the local principal unit group. Then for any set $\Sigma$ of $p$-adic primes of $k$, if we put $U_\Sigma^{(1)}\coloneqq\prod_{\pe\in \Sigma}U_\pe^{(1)}$, there is an exact sequence
\[
\begin{tikzpicture}
\node (b) at (0,0) {$\Z_p\otimes_\Z E(k)$}; 
\node (c) at (3,0) {$U_\Sigma^{(1)}$}; 
\node (d) at (5.7,0) {$X_{\Sigma}(k)$}; 
\node (e) at (8.4,0) {$X(k)$}; 
\node (f) at (10.2,0) {$0$};
\draw[->] (b) -- (c); \draw[->] (c) -- (d); \draw[->] (d) -- (e); \draw[->] (e) -- (f);
\end{tikzpicture}
\]
of $\Z_p$-modules by class field theory, where the leftmost morphism is obtained by taking $p$-adic completion of the diagonal embedding $E(k)\to\prod_{\pe\in\Sigma}U_\pe$. For a proof, see \cite[Lemma~1]{Fujii2017} for example. If the map $\Z_p\otimes_\Z E(k)\to U_\Sigma^{(1)}$ is injective, we say that $\Sigma$-adic Leopoldt's Conjecture holds for $k$. If $\Sigma=S_\Pe(k)$ holds for some prime $\Pe$ of a certain subfield of $k$, we use the term $\Pe$-adic Leopoldt's Conjecture for simplicity. If we let $U_\Sigma^{(1)}/(\Z_p\otimes_\Z E(k))$ denote the cokernel of the map $\Z_p\otimes_\Z E(k)\to U_\Sigma^{(1)}$ (even if the map is not injective), there is an exact sequence
\[
\begin{tikzpicture}
\node (b) at (0,0) {$0$}; 
\node (c) at (3,0) {$U_\Sigma^{(1)}/(\Z_p\otimes_\Z E(k))$}; 
\node (d) at (6.5,0) {$X_{\Sigma}(k)$}; 
\node (e) at (8.8,0) {$X(k)$}; 
\node (f) at (10.4,0) {$0$};
\draw[->] (b) -- (c); \draw[->] (c) -- (d); \draw[->] (d) -- (e); \draw[->] (e) -- (f);
\end{tikzpicture}
\]
of $\Z_p$-modules. So the $\Z_p$-ranks of $U_\Sigma^{(1)}/(\Z_p\otimes_\Z E(k))$ and $X_{\Sigma}(k)$ coincide. 

Here we remark that $\Sigma$-adic Leopoldt's Conjecture does not hold in general, but is believed to hold for $\Sigma=S_p(k)$, where it is the usual $p$-adic Leopoldt's Conjecture. For example, take $\Sigma$ small enough to satisfy $\sum_{\pe\in\Sigma}[k_\pe:\Q_p]<r_1(k)+r_2(k)-1$, then $\Sigma$-adic Leopoldt's Conjecture fails to hold. Here we let $r_1(k)$ (resp. $r_2(k)$) denote the number of real (resp. complex) places of $k$. On the other hand, if we let $K$ be the field $\Q$ or an imaginary quadratic field and $\Pe$ a $p$-adic prime of $K$, then $\Pe$-adic Leopoldt's Conjecture is known to hold for arbitrary finite abelian extension over $K$, by Brumer's result in \cite{Brumer1967}. We also remark that $\Sigma$-adic Leopoldt's Conjecture is stronger than $p$-adic Leopoldt's Conjecture in the sense of the following lemma. 
\begin{lem}\label{lem Leopoldt strong}
Let $k/\Q$ be a finite extension, $p$ a prime number, and $\Sigma$ a set of $p$-adic primes of $k$. If $\Sigma'$-adic Leopoldt's Conjecture holds for some subset $\Sigma'\subseteq \Sigma$, then $\Sigma$-adic Leopoldt's Conjecture also holds. 
\end{lem}
\begin{proof}
Consider a commutative diagram
\[
\begin{tikzpicture}
\node (a) at (0,2) {$\Z_p\otimes_\Z E(k)$}; 
\node (b) at (3.8,2) {$U_\Sigma^{(1)}$}; 
\node (d) at (3.8,0) {$U_{\Sigma'}^{(1)}$}; 
\draw[->] (a) --(b); \draw[->] (a) --(d); \draw[->] (b) --(d); 
\end{tikzpicture}
\]
of $\Z_p$-modules, where the vertical map is the natural projection. Then the injectivity of the diagonal map implies the injectivity of the horizontal map. 
\end{proof}
Now the $\Z_p$-rank of $\Z_p\otimes_\Z E(k)$ is $r_1(k)+r_2(k)-1$ and the $\Z_p$-rank of $U_\pe^{(1)}$ is $[k_\pe:\Q_p]$ for each $\pe\in\Sigma$. Hence we may calculate the $\Z_p$-rank of $X_\Sigma(k)$ by considering the above exact sequence and $\Sigma$-adic Leopoldt's Conjecture. In this way, we obtain the following lemma. 
\begin{lem}\label{lem two three}
Let $k/\Q$ be a finite extension, $p$ a prime number, and $\Sigma$ a set of $p$-adic primes of $k$. Consider the following conditions. 
\begin{enumerate}
\item $\Sigma$-adic Leopoldt's Conjecture holds for $k$. 

\item The $\Z_p$-rank of $X_\Sigma(k)$ is $1$. 

\item $\sum_{\pe\in\Sigma}[k_\pe:\Q_p]=r_1(k)+r_2(k)$ holds. 
\end{enumerate}
If two of the above three conditions hold, then the remaining one also holds. 
\end{lem}
\begin{proof}
This follows by comparing the $\Z_p$-ranks of the modules appearing in the exact sequence stated above. Here we note that $\Sigma$-adic Leopoldt's Conjecture is equivalent to the assertion that the $\Z_p$-rank of the kernel of the map $\Z_p\otimes_\Z E(k)\to U_\Sigma^{(1)}$ is $0$, unless $\Sigma$ is empty, since the map is injective on the $\Z_p$-torsion submodule of $\Z_p\otimes_\Z E(k)$. 
\end{proof}
The following gives another formulation of $\Sigma$-adic Leopoldt's Conjecture. 
\begin{lem}[{\cite[Theorem~37~(1)]{Maire2002}}]\label{lem Maire equiv}
Let $k/\Q$ be a finite extension, $p$ a prime number, and $\Sigma$ a non-empty set of $p$-adic primes of $k$. Then the following are equivalent. 
\begin{enumerate}[label=(\roman*)]
\item $\Sigma$-adic Leopoldt's Conjecture holds for $k$. 

\item $\Homo^2(G_\Sigma(k), \Q_p/\Z_p)$ is trivial, where $G_\Sigma(k)$ is the Galois group of the maximal $\Sigma$-ramified pro-$p$-extension over $k$. 
\end{enumerate}
\end{lem}
\begin{proof}
Note that the triviality of $\Homo^2(G_\Sigma(k), \Q_p/\Z_p)$ is equivalent to the triviality of its Pontryagin dual $\Homo_2(G_\Sigma(k), \Z_p)$. By \cite[Theorem~37~(1)]{Maire2002}, the $p$-rank of $\Homo_2(G_\Sigma(k), \Z_p)$ is equal to
\[
-\sum_{\pe\in\Sigma}[k_\pe:\Q_p]+r_1(k)+r_2(k)-1+\rank_{\Z_p}X_\Sigma(k)=\rank_{\Z_p}\Ker(\Z_p\otimes_\Z E(k)\to U_\Sigma^{(1)}). 
\]
Thus the assertion follows. 
\end{proof}

\section{Lemmas on the $\Sigma$-ramified $\Z_p$-extension}\label{section Sigma ramified}
Let $k/\Q$ be a finite extension, $p$ a prime number, and $\Sigma$ a set of $p$-adic primes of $k$. Suppose that the $\Z_p$-rank of $X_\Sigma(k)$ is $1$ and let $N_\infty/k$ be the unique $\Sigma$-ramified $\Z_p$-extension over $k$. Put $\Gamma\coloneqq\Gal(N_\infty/k)$ and $\Lambda\coloneqq\Z_p\bbr{\Gamma}$. In this section, we gather basic properties of the $\Z_p$-extension $N_\infty/k$. 

The following lemmas are well-known in the case where $k$ is totally real and $\Sigma=S_p(k)$. 
\begin{lem}\label{lem finite tors}
$X_\Sigma(N_\infty)_\Gamma$ and $X(N_\infty)_\Gamma$ are finite. Moreover, there is a natural isomorphism $X_\Sigma(N_\infty)_\Gamma\simeq\Tor_{\Z_p}X_\Sigma(k)$ of $\Z_p$-modules. 
\end{lem}
\begin{proof}
Since $\Gal(M_\Sigma(N_\infty)\cap k^\ab/k)=X_\Sigma(k)$ holds, there is an exact sequence
\[
\begin{tikzpicture}
\node (a) at (0,0) {$0$};
\node (b) at (2.3,0) {$X_\Sigma(N_\infty)_\Gamma$}; 
\node (c) at (5,0) {$X_\Sigma(k)$}; 
\node (d) at (7,0) {$\Gamma$}; 
\node (e) at (8.5,0) {$0$}; 
\draw[->] (a) -- (b); \draw[->] (b) -- (c); \draw[->] (c) -- (d); \draw[->] (d) -- (e); 
\end{tikzpicture}
\]
of $\Z_p$-modules by Lemma~\ref{lem five-term}. By the assumption $\rank_{\Z_p}X_\Sigma(k)=1$, it follows from the exact sequence that $X_\Sigma(N_\infty)_\Gamma$ is finite. Thus $X(N_\infty)_\Gamma$ is also finite. By taking the $\Z_p$-torsion submodules in the above exact sequence, we obtain the desired isomorphism $X_\Sigma(N_\infty)_\Gamma\simeq\Tor_{\Z_p}X_\Sigma(k)$ because $\Tor_{\Z_p}\Gamma$ is trivial. 
\end{proof}
The following can be seen as a $\Sigma$-adic variant of \cite[Proposition~1]{Greenberg1978}. 
\begin{lem}\label{lem fg tors}
$X_\Sigma(N_\infty)$ is finitely generated and torsion as a $\Lambda$-module. 
\end{lem}
\begin{proof}
Since $X_\Sigma(N_\infty)_\Gamma$ is finite by Lemma~\ref{lem finite tors}, we see that $X_\Sigma(N_\infty)$ is finitely generated as a $\Lambda$-module by Nakayama's Lemma and it is torsion as a $\Lambda$-module by Cayley--Hamilton Theorem. 
\end{proof}
For a finitely generated $\Lambda$-module $M$, we let $M_\fin$ denote the maximal finite $\Lambda$-submodule of $M$. 
\begin{lem}\label{lem invariant}
$X_\Sigma(N_\infty)^\Gamma=(X_\Sigma(N_\infty)_\fin)^\Gamma$ and $X(N_\infty)^\Gamma=(X(N_\infty)_\fin)^\Gamma$ hold. 
\end{lem}
\begin{proof}
Let $X$ be one of $X_\Sigma(N_\infty)$ or $X(N_\infty)$. Then $X_\Gamma$ is finite by Lemma~\ref{lem finite tors} and $X$ is a finitely generated torsion module over $\Lambda$ by Lemma~\ref{lem fg tors}. By an exact sequence
\[
\begin{tikzpicture}
\node (a) at (0,0) {$0$};
\node (b) at (2,0) {$X^\Gamma$}; 
\node (c) at (4.5,0) {$X$}; 
\node (d) at (7,0) {$X$}; 
\node (e) at (9.5,0) {$X_\Gamma$}; 
\node (f) at (11.5,0) {$0$}; 
\draw[->] (a) -- (b); \draw[->] (b) -- (c); \draw[->] (c) -- (d); \draw[->] (d) -- (e); \draw[->] (e) -- (f); 
\end{tikzpicture}
\]
of $\Lambda$-modules, we see that the characteristic ideal of $X^\Gamma$ is trivial, so $X^\Gamma$ is finite. Thus we obtain $X^\Gamma=(X_\fin)^\Gamma$. 
\end{proof}
\begin{lem}\label{lem isom injection}
Suppose that $U_\Sigma^{(1)}/(\Z_p\otimes_\Z E(k))$ is free as a $\Z_p$-module. Then the following hold. 
\begin{enumerate}
\item The natural surjection $X_\Sigma(N_\infty)\to X(N_\infty)$ induces an isomorphism $X_\Sigma(N_\infty)_\Gamma\simeq X(N_\infty)_\Gamma$ of $\Z_p$-modules. 

\item The natural morphism $X(N_\infty)\to X(k)$ induces an injection $X(N_\infty)_\Gamma\to X(k)$ of $\Z_p$-modules. 
\end{enumerate}
\end{lem}
\begin{proof}
Consider the exact sequence
\[
\begin{tikzpicture}
\node (b) at (1.5,0) {$0$};
\node (c) at (4.6,0) {$U_\Sigma^{(1)}/(\Z_p\otimes_\Z E(k))$}; 
\node (d) at (8,0) {$X_{\Sigma}(k)$}; 
\node (e) at (10.4,0) {$X(k)$}; 
\node (f) at (12,0) {$0$};
\draw[->] (b) -- (c); \draw[->] (c) -- (d); \draw[->] (d) -- (e); \draw[->] (e) -- (f);
\end{tikzpicture}
\]
of $\Z_p$-modules. By taking the $\Z_p$-torsion submodules in the exact sequence, we obtain a natural injection $\Tor_{\Z_p}X_\Sigma(k)\to X(k)$ because $\Tor_{\Z_p}\left(U_\Sigma^{(1)}/(\Z_p\otimes_\Z E(k))\right)$ is trivial by the assumption. Hence there is a natural injection $X_\Sigma(N_\infty)_\Gamma\to X(k)$ by Lemma~\ref{lem finite tors}. Now consider a commutative diagram
\[
\begin{tikzpicture}
\node (a) at (0,2) {$X_\Sigma(N_\infty)_\Gamma$}; 
\node (b) at (3.7,2) {$X(N_\infty)_\Gamma$}; 
\node (d) at (3.7,0) {$X(k)$}; 
\draw[->] (a) --(b); \draw[->] (a) --(d); \draw[->] (b) --(d); 
\end{tikzpicture}
\]
of $\Z_p$-modules. Since the diagonal map is injective, it follows that the horizontal surjection is an isomorphism. Thus it also follows that the vertical map is injective. 
\end{proof}
The arguments in the proof of this lemma can be found in the proof of \cite[Theorem~2]{Fujii2020}, where it is assumed (in our notation) that $\#\Sigma=1$ and $N_\infty/k$ is totally ramified at the unique prime in $\Sigma$. Under the assumption, there is a natural isomorphism $X(N_\infty)_\Gamma\simeq X(k)$ by \cite[Proposition~13.22]{Washington1997}, and this is used to show the isomorphism $X_\Sigma(N_\infty)_\Gamma\simeq X(N_\infty)_\Gamma$. However, as in the proof of Lemma~\ref{lem isom injection}, the injectivity of the map $X_\Sigma(N_\infty)_\Gamma\to X(k)$ suffices for the purpose. 

Although the following proposition is elementary and useful, we could not find an explicit reference, so we record it here. 
\begin{prop}\label{prop trivial}
Suppose that $N_\infty/k$ is totally ramified at some $p$-adic prime and $U_\Sigma^{(1)}/(\Z_p\otimes_\Z E(k))$ is free as a $\Z_p$-module. Then there is a natural isomorphism $X(N_\infty)_{\Gamma}\simeq X(k)$ of $\Z_p$-modules and, in particular, the following are equivalent. 
\begin{enumerate}[label=(\roman*)]
\item $X(N_\infty)$ is trivial. 

\item $p$ does not divide the class number of $k$. 
\end{enumerate}
\end{prop}
\begin{proof}
The map $X(N_\infty)_\Gamma\to X(k)$ is injective by Lemma~\ref{lem isom injection}. Since $N_\infty/k$ is totally ramified at some $p$-adic prime, the map is surjective. 
\end{proof}
The following proposition concerns an analogue of weak Leopoldt's Conjecture. We do not need the proposition in the arguments in what follows, but we state it here for its own interest. Here we recall that $N_n$ denotes the unique intermediate field of $N_\infty/k$ such that $[N_n:k]=p^n$. 
\begin{prop}\label{prop WLC}
Put $\Sigma_n\coloneqq S_\Sigma(N_n)$ and let $\delta_n$ be the $\Z_p$-rank of the kernel of the map $\Z_p\otimes_\Z E(N_n)\to U_{\Sigma_n}^{(1)}$ for each $n\geq 0$. Then the following are equivalent. 
\begin{enumerate}[label=(\roman*)]
\item $\Sigma$-adic Leopoldt's Conjecture holds for $k$. 

\item $\{\delta_n\mid n\geq 0\}$ is bounded. 
\end{enumerate}
\end{prop}
\begin{proof}
Our proof is based on the argument in \cite[\S3]{Greenberg1978}. For each $n\geq 0$, since $M_\Sigma(N_n)=M_\Sigma(N_\infty)\cap N_n^\ab$ holds, we have 
\[
\rank_{\Z_p}X_\Sigma(N_n)=\rank_{\Z_p}X_\Sigma(N_\infty)_{\Gal(N_\infty/N_n)}+1\leq\rank_{\Z_p}X_\Sigma(N_\infty)+1
\]
by Lemma~\ref{lem five-term}. Thus it follows that $\{\rank_{\Z_p}X_\Sigma(N_n)\mid n\geq 0\}$ is bounded, because $\rank_{\Z_p}X_\Sigma(N_\infty)$ is finite by Lemma~\ref{lem fg tors}. Now consider the exact sequence
\[
\begin{tikzpicture}
\node (b) at (2,0) {$\Z_p\otimes_\Z E(N_n)$}; 
\node (c) at (5.1,0) {$U_{\Sigma_n}^{(1)}$}; 
\node (d) at (8,0) {$X_{\Sigma}(N_n)$}; 
\node (e) at (10.7,0) {$X(N_n)$}; 
\node (f) at (13,0) {$0$};
\draw[->] (b) -- (c); \draw[->] (c) -- (d); \draw[->] (d) -- (e); \draw[->] (e) -- (f);
\end{tikzpicture}
\]
of $\Z_p$-modules. Here we have 
\[
\rank_{\Z_p}(\Z_p\otimes_\Z E(N_n))=r_1(k)p^n+r_2(k)p^n-1=p^n(r_1(k)+r_2(k))-1, 
\]
and
\[
\rank_{\Z_p}U_{\Sigma_n}^{(1)}=\sum_{\pe\in\Sigma_n}[(N_n)_\pe:\Q_p]=p^n\sum_{\pe\in\Sigma}[k_\pe:\Q_p]
\]
holds. Therefore, we obtain
\[
\delta_n=\rank_{\Z_p}X_\Sigma(N_n)+p^n\left( r_1(k)+r_2(k)-\sum_{\pe\in\Sigma}[k_\pe:\Q_p] \right)-1. 
\]
Hence we see that $\{\delta_n\mid n\geq 0\}$ is bounded if and only if $\sum_{\pe\in\Sigma}[k_\pe:\Q_p]=r_1(k)+r_2(k)$ holds, which is equivalent to $\Sigma$-adic Leopoldt's Conjecture by Lemma~\ref{lem two three}. 
\end{proof}
The following is a $\Sigma$-adic variant of \cite[Proposition~3]{Greenberg1978}. 
\begin{prop}\label{lem invariant trivial}
Suppose that $\Sigma$-adic Leopoldt's Conjecture holds for $k$. Then $X_\Sigma(N_\infty)^\Gamma$ is trivial and $X_\Sigma(N_\infty)$ has no non-trivial finite $\Lambda$-submodule. 
\end{prop}
\begin{proof}
One may show this proposition by tracing the proof of \cite[Proposition~3]{Greenberg1978}, but we give a simpler proof. 

Let $G_\Sigma(k)$ (resp. $G_\Sigma(N_\infty)$) be the Galois group of the maximal $\Sigma$-ramified pro-$p$-extension over $k$ (resp. $N_\infty$). Then $G_\Sigma(N_\infty)$ is a normal subgroup of $G_\Sigma(k)$ such that $G_\Sigma(k)/G_\Sigma(N_\infty)=\Gamma$. Since $\Homo^i(\Gamma, \Q_p/\Z_p)$ is trivial for $i=2,3$, we obtain an isomorphism
\[
\Ker\left(\Homo^2(G_\Sigma(k), \Q_p/\Z_p)\to \Homo^2(G_\Sigma(N_\infty), \Q_p/\Z_p)\right)\simeq\Homo^1(\Gamma, \Homo^1(G_\Sigma(N_\infty), \Q_p/\Z_p)) 
\]
by Hochschild--Serre spectral sequence (see \cite[Theorem~7.2.4]{RZ2010}, for example). The left hand side is trivial by Lemma~\ref{lem Maire equiv}, so the right hand side is trivial as well. Observe that
\[
\Homo^1(\Gamma, \Homo^1(G_\Sigma(N_\infty), \Q_p/\Z_p)) \simeq (X_\Sigma(N_\infty)^{\lor})_\Gamma\simeq (X_\Sigma(N_\infty)^\Gamma)^{\lor}
\]
holds, so $X_\Sigma(N_\infty)^{\Gamma}$ is trivial. Here $(\cdot)^{\lor}$ denotes the Pontryagin dual. Hence, by Lemma~\ref{lem invariant}, we see that $(X_\Sigma(N_\infty)_\fin)^\Gamma$ is trivial. Therefore, we conclude that $X_\Sigma(N_\infty)_\fin$ is trivial. 
\end{proof}

\section{Proof of Theorems~\ref{thm main GC} and \ref{thm WGC 1}}\label{section 5}
In this section, we prove Theorems~\ref{thm main GC} and \ref{thm WGC 1}. Our proof is based on Fujii's proof of \cite[Theorem~1]{Fujii2020}. We use the following lemma. 
\begin{lem}[{\cite[Proposition]{Ozaki1995}}]\label{lem capitulation}
Let $k/\Q$ be a finite extension, $p$ a prime number, and $k_\infty/k$ a $\Z_p$-extension. Suppose that $k_\infty/k$ is totally ramified at every ramified prime. Then there is an isomorphism
\[
\Ker(i_{0, \infty}\colon A(k)\to A(k_\infty))\simeq\Image(X(k_\infty)_{\mathrm{fin}}\to X(k))
\]
of $\Z_p$-modules. Here, the isomorphism is induced by the Artin map. 
\end{lem}
We deduce the theorems from the following general result. 
\begin{thm}\label{thm A and B}
Let $k/\Q$ be a finite extension, $p$ a prime number, and $\Sigma$ a set of $p$-adic primes of $k$. Let $N_\infty/k$ be a $\Sigma$-ramified $\Z_p$-extension that is totally ramified at every ramified prime. Suppose also that one of the following two conditions holds. 
\begin{enumerate}[label=(\Alph*)]
\item $\Sigma$ consists of a single prime. \label{A}

\item $U_\Sigma^{(1)}/(\Z_p\otimes_\Z E(k))$ is a free $\Z_p$-module of rank $1$. \label{B}
\end{enumerate}
Then there is a natural isomorphism
\[
(X(N_\infty)_\fin)_{\Gal(N_\infty/k)}\simeq \Ker(i_{0,\infty}\colon A(k)\to A(N_\infty))
\]
of $\Z_p$-modules. 
\end{thm}
Now suppose the assumptions of Theorem~\ref{thm A and B}. Put $\Gamma\coloneqq\Gal(N_\infty/k)$ and $\Lambda\coloneqq\Z_p\bbr{\Gamma}$. 
\begin{proof}[Proof of Theorem~\ref{thm A and B}]
Consider an exact sequence
\[
\begin{tikzpicture}
\node (a) at (0,0) {$0$};
\node (b) at (2.2,0) {$X(N_\infty)_\fin$}; 
\node (c) at (5,0) {$X(N_\infty)$}; 
\node (d) at (8.3,0) {$X(N_\infty)/X(N_\infty)_\fin$}; 
\node (e) at (11.1,0) {$0$}; 
\draw[->] (a) -- (b); \draw[->] (b) -- (c); \draw[->] (c) -- (d); \draw[->] (d) -- (e); 
\end{tikzpicture}
\]
of $\Lambda$-modules. Here we recall that there is a natural isomorphism $X(N_\infty)_\Gamma\simeq X(k)$ by \cite[Proposition~13.22]{Washington1997} (in Case~\ref{A}) and Proposition~\ref{prop trivial} (in Case~\ref{B}). Since $X(N_\infty)_\Gamma$ is finite, $(X(N_\infty)/X(N_\infty)_\fin)_\Gamma$ is also finite. Hence $(X(N_\infty)/X(N_\infty)_\fin)^\Gamma$ is finite as well, so we obtain $(X(N_\infty)/X(N_\infty)_\fin)^\Gamma=0$ because $X(N_\infty)/X(N_\infty)_\fin$ has no non-trivial finite $\Lambda$-submodule. Therefore, by Snake Lemma, we obtain a natural injection $(X(N_\infty)_\fin)_\Gamma\to X(k)$. Since the image of this injection coincides with $\Image(X(N_\infty)_\fin\to X(k))$, we obtain the assertion by Lemma~\ref{lem capitulation}. 
\end{proof}
\begin{proof}[Proof of Theorem~\ref{thm WGC 1}]
By Theorem~\ref{thm A and B}, the map $i_{0, \infty}$ is injective if and only if $(X(N_\infty)_\fin)_\Gamma=0$ holds. By Nakayama's Lemma, this condition is equivalent to the assertion $X(N_\infty)_\fin=0$, that is, $X(N_\infty)$ has no non-trivial finite $\Lambda$-submodule. 
\end{proof}
\begin{rem}\label{remark kernel finite}
One can deduce information from Theorem~\ref{thm A and B} not only of the triviality of $X(N_\infty)_{\mathrm{fin}}$ but also of the minimum number of generators of $X(N_\infty)_{\mathrm{fin}}$ as a $\Lambda$-module. Indeed, the number coincides with the $p$-rank of $\Ker(i_{0, \infty})$ by Nakayama's Lemma. Moreover, the $p$-rank of $\Ker(i_{0,n})$ gives a lower bound for each $n\geq 0$. 
\end{rem}
\begin{proof}[Proof of Theorem~\ref{thm main GC}]
By Theorem~\ref{thm A and B}, the map $i_{0, \infty}$ is trivial if and only if $(X(N_\infty)_\fin)_\Gamma\simeq X(k)$ holds. Since there is a natural isomorphism $X(N_\infty)_\Gamma\simeq X(k)$, we have $(X(N_\infty)_\fin)_\Gamma\simeq X(k)$ if and only if $(X(N_\infty)_\fin)_\Gamma\simeq X(N_\infty)_\Gamma$ holds. By Nakayama's Lemma, this condition is equivalent to the assertion $X(N_\infty)=X(N_\infty)_\fin$, that is, $X(N_\infty)$ is finite. 
\end{proof}

\section{Proof of Theorems~\ref{thm main WGC 2}, \ref{thm main WGC 3}, and \ref{thm triviality}}\label{section 6}
In this section, we prove Theorems~\ref{thm main WGC 2}, \ref{thm main WGC 3}, and \ref{thm triviality}. The following is a key lemma, which is a  partial generalization of \cite[Proposition~1]{Ozaki1997} (see \cite[Lemma~4]{Jaulent2024} as well). 
\begin{lem}\label{lem ur}
Let $k/\Q$ be a finite extension, $p$ an odd prime number, and $\Sigma$ a set of $p$-adic primes of $k$ that are of degree $1$. Let $N_\infty/k$ be a $\Sigma$-ramified $\Z_p$-extension and suppose that it is ramified at $\Sigma$. Then $M_\Sigma(k)\subseteq L(N_\infty)$ holds and, in particular, $M_\Sigma(k)=L(N_\infty)\cap k^\ab$ holds. 
\end{lem}
\begin{proof}
This lemma can be deduced from a part of the proof of \cite[Proposition~1]{Ozaki1997}, but we prefer to prove here, which is a direct generalization of the original one. 

First note that $M_\Sigma(k)/N_\infty$ is a $\Sigma$-ramified abelian pro-$p$-extension. Thus it is sufficient to show that $M_\Sigma(k)/N_\infty$ is unramified at every prime lying above $\Sigma$. Take a prime $\pe\in\Sigma$ and let $I_\pe$ denote the inertia group of $\pe$ in $M_\Sigma(k)/k$. Then, by local class field theory, there is a surjection on $I_\pe$ from the pro-$p$-quotient of $U_\pe$, which is isomorphic to $\Z_p$ because $p$ is odd and $\pe$ is of degree $1$. Since $N_\infty/k$ is infinitely ramified at $\pe$, we have $I_\pe\simeq\Z_p$ and $I_\pe$ injects into $\Gamma$. Hence we obtain $I_\pe\cap\Gal(M_\Sigma(k)/N_\infty)=0$. Therefore, $M_\Sigma(k)/N_\infty$ is unramified at every prime lying above $\pe$. This completes the proof. 
\end{proof}
We use the notation $\Gamma\coloneqq\Gal(N_\infty/k)$ and $\Lambda\coloneqq\Z_p\bbr{\Gamma}$, as usual.  
\begin{lem}\label{lem isom}
Suppose the assumptions of Lemma~\ref{lem ur}. Then the natural surjection $X_\Sigma(N_\infty)\to X(N_\infty)$ induces an isomorphism $X_\Sigma(N_\infty)_\Gamma\simeq X(N_\infty)_\Gamma$ of $\Z_p$-modules. 
\end{lem}
\begin{proof}
By using Lemma~\ref{lem ur}, we obtain 
\[
X_\Sigma(N_\infty)_\Gamma=\Gal(M_\Sigma(N_\infty)\cap k^\ab/N_\infty)=\Gal(M_\Sigma(k)/N_\infty)=\Gal(L(N_\infty)\cap k^\ab/N_\infty)=X(N_\infty)_\Gamma,
\]
as desired. 
\end{proof}
The following is a generalization of \cite[Proposition~2]{Ozaki1997}. See \cite[Lemma~2.1]{Nguyen2006} as well. 
\begin{prop}\label{prop main WGC}
Suppose the assumptions of Theorem~\ref{thm main WGC 2} or \ref{thm main WGC 3}. Then there is an isomorphism 
\[
\Gal(M_\Sigma(N_\infty)/L(N_\infty))_\Gamma\simeq (X(N_\infty)_{\mathrm{fin}})^\Gamma
\]
of $\Z_p$-modules. 
\end{prop}
\begin{proof}
Consider an exact sequence
\[
\begin{tikzpicture}
\node (a) at (0,0) {$0$};
\node (b) at (3,0) {$\Gal(M_\Sigma(N_\infty)/L(N_\infty))$}; 
\node (c) at (6.5,0) {$X_\Sigma(N_\infty)$}; 
\node (d) at (8.7,0) {$X(N_\infty)$}; 
\node (e) at (10.5,0) {$0$}; 
\draw[->] (a) -- (b); \draw[->] (b) -- (c); \draw[->] (c) -- (d); \draw[->] (d) -- (e); 
\end{tikzpicture}
\]
of $\Lambda$-modules. Then we obtain an exact sequence
\[
\begin{tikzpicture}
\node (a) at (0,0) {$X_\Sigma(N_\infty)^{\Gamma}$};
\node (b) at (2.4,0) {$X(N_\infty)^{\Gamma}$}; 
\node (c) at (6,0) {$\Gal(M_\Sigma(N_\infty)/L(N_\infty))_{\Gamma}$}; 
\node (d) at (9.7,0) {$X_\Sigma(N_\infty)_{\Gamma}$}; 
\node (e) at (12.3,0) {$X(N_\infty)_{\Gamma}$}; 
\draw[->] (a) -- (b); \draw[->] (b) -- (c); \draw[->] (c) -- (d); \draw[->] (d) -- (e); 
\end{tikzpicture}
\]
of $\Z_p$-modules by Snake Lemma. The rightmost morphism is an isomorphism by Lemmas~\ref{lem isom} (for Theorem~\ref{thm main WGC 2}) and \ref{lem isom injection} (for Theorem~\ref{thm main WGC 3}). On the other hand, $X_\Sigma(N_\infty)^\Gamma$ is trivial by Proposition~\ref{lem invariant trivial}. Therefore, we obtain an isomorphism $X(N_\infty)^{\Gamma}\simeq\Gal(M_\Sigma(N_\infty)/L(N_\infty))_{\Gamma}$. Since $X(N_\infty)^\Gamma=(X(N_\infty)_{\mathrm{fin}})^\Gamma$ holds by Lemma~\ref{lem invariant}, we obtain the desired isomorphism. 
\end{proof}
Now we are ready to finish the proof of Theorems~\ref{thm main WGC 2} and \ref{thm main WGC 3}. 
\begin{proof}[Proof of Theorems~\ref{thm main WGC 2} and \ref{thm main WGC 3}]
If $L(N_\infty)=M_\Sigma(N_\infty)$ holds, then $X(N_\infty)=X_\Sigma(N_\infty)$, so $X(N_\infty)$ has no non-trivial finite $\Lambda$-submodule by Proposition~\ref{lem invariant trivial}. 

If $L(N_\infty)\neq M_\Sigma(N_\infty)$ holds, then we have $\Gal(M_\Sigma(N_\infty)/L(N_\infty))_\Gamma\neq 0$ by Nakayama's Lemma, so we obtain $(X(N_\infty)_{\mathrm{fin}})^{\Gamma}\neq 0$ by Proposition~\ref{prop main WGC} and, in particular, we have $X(N_\infty)_{\mathrm{fin}}\neq 0$. This completes the proof of Theorems~\ref{thm main WGC 2} and \ref{thm main WGC 3}. 
\end{proof}
As an application of the above arguments, we prove Theorem~\ref{thm triviality}. 
\begin{proof}[Proof of Theorem~\ref{thm triviality}]
Suppose that \ref{one} holds. Then $X(k)$ is trivial because there is a natural surjection $X(N_\infty)\to X(k)$ since $N_\infty/k$ is totally ramified at some prime in $\Sigma$. On the other hand, by Lemma~\ref{lem ur}, we have $M_\Sigma(k)=N_\infty$. So $X_\Sigma(k)$ is isomorphic to $\Z_p$. Therefore, we conclude that $U_\Sigma^{(1)}/(\Z_p\otimes_\Z E(k))$ is isomorphic to $\Z_p$. 

On the other hand, we see that \ref{two} implies \ref{one} by Proposition~\ref{prop trivial}. Alternatively, we can show the implication as follows. If we suppose that \ref{two} holds, then $X_\Sigma(k)$ is isomorphic to $\Z_p$, so $M_\Sigma(k)$ coincides with $N_\infty$. Thus we have $X(N_\infty)_\Gamma=\Gal(L(N_\infty)\cap k^\ab/N_\infty)=\Gal(M_\Sigma(k)/N_\infty)=0$ by Lemma~\ref{lem ur}. Hence $X(N_\infty)$ is trivial by Nakayama's Lemma. 
\end{proof}
Here we give a note on the relation with Fukuda--Komatsu's result \cite[Theorem~1, Corollary]{FuKo1986}. Their result shows the vanishing of the $\mu$- and $\lambda$-invariants, and gives the explicit value of the $\nu$-invariant of the cyclotomic $\Z_p$-extension of a real quadratic field. If we focus on the case where the $\nu$-invariant is $0$, we see that the following holds by their result. 
\begin{cor}\label{cor FK}
Let $k$ be a real quadratic field and $p$ an odd prime number that splits in $k$. Let $\pe$ be a $p$-adic prime of $k$ and $\varepsilon$ a fundamental unit of $k$. Then the following are equivalent. 
\begin{enumerate}[label=(\roman*)]
\item $X(k^\cy_\infty)$ is trivial. 

\item $p$ does not divide the class number of $k$ and $\varepsilon^{p-1}\not\equiv 1\pmod{\pe^2}$ holds. 
\end{enumerate}
\end{cor}
The condition $\varepsilon^{p-1}\not\equiv 1\pmod{\pe^2}$ is equivalent to the freeness of $U_{S_p(k)}^{(1)}/(\Z_p\otimes_\Z E(k))$, so Corollary~\ref{cor FK} follows from Theorem~\ref{thm triviality}.

\section{Examples}\label{section Examples}
In this section, we show some examples of $\Sigma$-adic Leopoldt's Conjecture. 
\begin{condition}\label{condition basic}
Let $k/\Q$ be a finite extension, $p$ a prime number, and $\Sigma$ a set of $p$-adic primes of $k$ that are of degree $1$. Suppose that the $\Z_p$-rank of $X_\Sigma(k)$ is $1$ and the unique $\Sigma$-ramified $\Z_p$-extension $N_\infty/k$ is ramified at $\Sigma$. Suppose also that $\Sigma$-adic Leopoldt's Conjecture holds for $k$. 
\end{condition}
If Condition~\ref{condition basic} is satisfied, then we say that $(k, p, \Sigma)$ satisfies Condition~\ref{condition basic}. If $(k, p, \Sigma)$ satisfies Condition~\ref{condition basic} and $p$ is odd, then Theorem~\ref{thm main WGC 2} can be applied in this setting. 
\begin{lem}\label{lem basic}
Let $k/\Q$ be a finite extension, $p$ an odd prime number, and $\Sigma$ a set of $p$-adic primes of $k$. Then $(k, p, \Sigma)$ satisfies Condition~\ref{condition basic} if and only if the following two conditions are satisfied. 
\begin{enumerate}
\item $\sum_{\pe\in\Sigma}[k_\pe:\Q_p]=r_1(k)+r_2(k)$ holds. 

\item $(\Sigma\setminus\{\qu\})$-adic Leopoldt's Conjecture holds for every $\qu\in\Sigma$. 
\end{enumerate}
\end{lem}
\begin{proof}
Suppose that the above two conditions are satisfied. First note that $\Sigma$-adic Leopoldt's Conjecture follows from $(\Sigma\setminus\{\qu\})$-adic Leopoldt's Conjecture for any $\qu\in\Sigma$, by Lemma~\ref{lem Leopoldt strong}. Hence we obtain $\rank_{\Z_p}X_\Sigma(k)=1$ by Lemma~\ref{lem two three}. For each $\qu\in\Sigma$, there is an exact sequence
\[
\begin{tikzpicture}
\node (a) at (0,0) {$0$};
\node (b) at (2.2,0) {$\Z_p\otimes_\Z E(k)$}; 
\node (c) at (5.1,0) {$U_{\Sigma\setminus\{\qu\}}^{(1)}$}; 
\node (d) at (7.9,0) {$X_{\Sigma\setminus\{\qu\}}(k)$}; 
\node (e) at (10.4,0) {$X(k)$}; 
\node (f) at (12,0) {$0$};
\draw[->] (a) -- (b); \draw[->] (b) -- (c); \draw[->] (c) -- (d); \draw[->] (d) -- (e); \draw[->] (e) -- (f);
\end{tikzpicture}
\]
of $\Z_p$-modules. By the injectivity of the leftmost map, we obtain
\[
r_1(k)+r_2(k)-1\leq r_1(k)+r_2(k)-[k_\qu:\Q_p], 
\]
so $[k_\qu:\Q_p]=1$. It follows from the above exact sequence that the $\Z_p$-rank of $X_{\Sigma\setminus\{\qu\}}(k)$ is $0$. This means that $M_\Sigma(k)/k$ is infinitely ramified at $\qu$, so $N_\infty/k$ is ramified at $\qu$. 

Next suppose that $(k, p, \Sigma)$ satisfies Condition~\ref{condition basic}. Then $\sum_{\pe\in\Sigma}[k_\pe:\Q_p]=r_1(k)+r_2(k)$ holds by Lemma~\ref{lem two three}. Let $\qu\in\Sigma$. Since $N_\infty/k$ is ramified at $\qu$, the $\Z_p$-rank of $X_{\Sigma\setminus\{\qu\}}(k)$ is $0$. On the other hand, since $\qu$ is of degree $1$, the $\Z_p$-rank of $U_{\Sigma\setminus\{\qu\}}^{(1)}$ is $r_1(k)+r_2(k)-1$. Thus we obtain the injectivity of the map $\Z_p\otimes_\Z E(k)\to U_{\Sigma\setminus\{\qu\}}^{(1)}$. This completes the proof. 
\end{proof}
By using Lemma~\ref{lem basic}, we can show some natural classes of triples $(k, p, \Sigma)$ that satisfy Condition~\ref{condition basic} and, in particular, we can give many examples that satisfy the assumptions of Theorem~\ref{thm main WGC 2}.

\subsection{Fields in which $p$ splits completely}\label{subsection splits completely}
\begin{prop}\label{prop existence}
Let $k/\Q$ be a finite extension and $p$ a prime number that splits completely in $k$. Suppose that $p$-adic Leopoldt's Conjecture is true for $k$. Then there exists a set $\Sigma$ of $p$-adic primes of $k$ such that $\Sigma$-adic Leopoldt's Conjecture holds for $k$ and the $\Z_p$-rank of $X_\Sigma(k)$ is $1$. 
\end{prop}
\begin{proof}
We claim that for any set $S$ of $p$-adic primes of $k$, if the $\Z_p$-rank of $X_S(k)$ is positive, then there exists a prime $\qu\in S$ such that the $\Z_p$-rank of the inertia group of $\qu$ in $M_S(k)/k$ is $1$. Indeed, the $\Z_p$-rank of the inertia group of every $\qu\in S$ is at most $1$ by local class field theory, and the inertia groups for all $\qu\in S$ generate $\Gal(M_S(k)/L(k))$, which has positive $\Z_p$-rank by the assumption. In this setting, we have $\rank_{\Z_p}X_{S\setminus\{\qu\}}(k)=\rank_{\Z_p}X_S(k)-1$. 

Since we assume $p$-adic Leopoldt's Conjecture for $k$, the $\Z_p$-rank of $X_p(k)$ is $r_2(k)+1$, so we can inductively use the above claim to take pairwise different primes $\qu_1, \ldots, \qu_{r_2(k)}\in S_p(k)$ such that the $\Z_p$-rank of $X_\Sigma(k)$ is $1$,  where we put $\Sigma=S_p(k)\setminus\{\qu_1, \ldots, \qu_{r_2(k)}\}$. On the other hand, the validity of $\Sigma$-adic Leopoldt's Conjecture follows from Lemma~\ref{lem two three}. 
\end{proof}
This proposition only shows the existence of $\Sigma$, but we can show more explicit examples in the following subsections. See \S\ref{subsection conditional} for conditional deeper results on the existence of $\Sigma$.

\subsection{CM-fields}\label{subsection CM-fields}
Let $k$ be a CM-field and $p$ a prime number. Suppose that every $p$-adic prime of $k^+$ splits in $k$, where $k^+$ is the maximal totally real subfield of $k$. In this setting, a set $\Sigma$ of $p$-adic primes of $k$ is called a $p$-adic CM-type if it satisfies $\Sigma\cup\overline{\Sigma}=S_p(k)$ and $\Sigma\cap\overline{\Sigma}=\emptyset$. Here we put $\overline{\Sigma}\coloneqq\{\overline{\pe}\in S_p(k)\mid \pe\in\Sigma\}$, where $\overline{\pe}$ is the complex conjugate of $\pe$. We let $\Sigma$ be a $p$-adic CM-type in what follows. 
\begin{lem}\label{lem p and Sigma}
In the above setting, the following are equivalent. 
\begin{itemize}
\item[(i)] $\Sigma$-adic Leopoldt's Conjecture holds for $k$. 

\item[(ii)] $p$-adic Leopoldt's Conjecture is true for $k$. 
\end{itemize}
\end{lem}
\begin{proof}
For each $\pe\in S_p(k)$, we put $\pe_+\coloneqq\pe\cap k^+$. Then the correspondence $\pe\mapsto \pe_+$ defines a bijection $\Sigma\to S_p(k^+)$. Since every $p$-adic prime of $k^+$ splits in $k$, there is a natural isomorphism $U_\pe^{(1)}\simeq U_{\pe_+}^{(1)}$ for every $\pe\in\Sigma$. Thus there is a natural isomorphism $U_\Sigma^{(1)}\simeq U_{S_p(k^+)}^{(1)}$. Therefore, $\Sigma$-adic Leopoldt's Conjecture for $k$ is equivalent to $p$-adic Leopoldt's Conjecture for $k^+$, so it is equivalent to $p$-adic Leopoldt's Conjecture for $k$. 
\end{proof}
The following can be found, for example, in \cite[Theorem~2.1]{PS2024}, but we give a proof, because it is not difficult at this point. 
\begin{prop}\label{prop CM unique}
Suppose that $p$-adic Leopoldt's Conjecture is true for $k$. Then there exists a unique $\Sigma$-ramified $\Z_p$-extension over $k$. 
\end{prop}
\begin{proof}
Since $\Sigma$-adic Leopoldt's Conjecture holds by Lemma~\ref{lem p and Sigma} and
\[
\sum_{\pe\in \Sigma}[k_\pe:\Q_p]=\frac{1}{2}\sum_{\pe\in S_p(k)}[k_\pe:\Q_p]=\frac{1}{2}[k:\Q]=r_2(k)=r_1(k)+r_2(k)
\]
holds, the assertion follows from Lemma~\ref{lem two three}. 
\end{proof}
The following can be obtained by using \cite[Lemma~2]{Fujii2017}, but we give a proof here from our perspective. 
\begin{prop}\label{prop CM-type}
Let $k$ be a CM-field, $p$ an odd prime number that splits completely in $k$, and $\Sigma$ a $p$-adic CM-type. Assume that $p$-adic Leopoldt's Conjecture holds for $k$. Then $(k, p, \Sigma)$ satisfies Condition~\ref{condition basic}. 
\end{prop}
\begin{proof}
Let $\qu\in\Sigma$ and put $\qu_+\coloneqq\qu\cap k^+$. Consider a commutative diagram
\[
\begin{tikzpicture}
\node (a) at (0,2) {$\Z_p\otimes_\Z E(k^+)$}; 
\node (b) at (4,2) {$U_{S_p(k^+)\setminus\{\qu_+\}}^{(1)}$}; 
\node (c) at (0,0) {$\Z_p\otimes_\Z E(k)$}; 
\node (d) at (4,0) {$U_{\Sigma\setminus\{\qu\}}^{(1)}$}; 
\draw[->] (a) --(b); \draw[->] (a) --(c); \draw[->] (b) --(d); \draw[->] (c) -- (d); 
\end{tikzpicture}
\]
of $\Z_p$-modules, where the left vertical map is induced by the inclusion $E(k^+)\to E(k)$ and the right vertical map is defined by gathering the inclusions $U_{\pe_+}^{(1)}\to U_\pe^{(1)}$ for $\pe\in \Sigma\setminus\{\qu\}$ and $\pe_+=\pe\cap k^+$. Observe that both vertical maps are isomorphisms since $p$ splits completely in $k$. Since the cyclotomic $\Z_p$-extension over $k^+$ is ramified at $\qu_+$, the $\Z_p$-rank of $X_{S_p(k^+)\setminus\{\qu_+\}}(k^+)$ is $0$. Hence, by the exact sequence
\[
\begin{tikzpicture}
\node (b) at (0,0) {$\Z_p\otimes_\Z E(k^+)$}; 
\node (c) at (3.9,0) {$U_{S_p(k^+)\setminus\{\qu_+\}}^{(1)}$}; 
\node (d) at (8,0) {$X_{S_p(k^+)\setminus\{\qu_+\}}(k^+)$}; 
\node (e) at (11.5,0) {$X(k^+)$}; 
\node (f) at (13.7,0) {$0$};
\draw[->] (b) -- (c); \draw[->] (c) -- (d); \draw[->] (d) -- (e); \draw[->] (e) -- (f);
\end{tikzpicture}
\]
of $\Z_p$-modules, we obtain the injectivity of the top horizontal map in the above commutative diagram, so we obtain the injectivity of the bottom horizontal map. Therefore, we obtain the assertion by Lemma~\ref{lem basic}. 
\end{proof}
In the case where $k$ is an imaginary quadratic field, we may apply Theorem~\ref{thm main WGC 3} to obtain the following theorem. 
\begin{thm}\label{thm imag quad}
Let $k$ be an imaginary quadratic field, $p$ a prime number that splits in $k$, $\pe$ a $p$-adic prime of $k$, and $N_\infty/k$ the $\pe$-ramified $\Z_p$-extension. Then the following are equivalent. 
\begin{itemize}
\item[(i)] $X(N_\infty)$ has no non-trivial finite $\Z_p\bbr{\Gal(N_\infty/k)}$-submodule. 

\item[(ii)] $L(N_\infty)=M_\pe(N_\infty)$ holds. 
\end{itemize}
\end{thm}
\begin{proof}
First consider the case where $p$ is odd. Then $\Z_p\otimes_\Z E(k)$ is trivial, so $U_\pe^{(1)}/(\Z_p\otimes_\Z E(k))=U_\pe^{(1)}\simeq \Z_p$ holds. Thus we can apply Theorem~\ref{thm main WGC 3}. Alternatively, we can apply Theorem~\ref{thm main WGC 2} to obtain the assertion. 

Next consider the case where $p=2$. Since $U_\pe^{(1)}\simeq 1+2\Z_2\simeq \Z_2\oplus \Z/2\Z$ and $\Z_2\otimes_\Z E(k)\simeq \Z/2\Z$ hold, we have $U_\pe^{(1)}/(\Z_2\otimes_\Z E(k))\simeq\Z_2$. Hence we can apply Theorem~\ref{thm main WGC 3}. 
\end{proof}
\begin{rem}
In the setting of Theorem~\ref{thm imag quad}, Fujii showed a similar result \cite[Theorem~2]{Fujii2014} for the $\Z_p^2$-extension $\widetilde{k}/k$. Fujii's result asserts that, if $p$ is odd, then $X(\widetilde{k})$ has no non-trivial pseudo-null $\Z_p\bbr{\Gal(\widetilde{k}/k)}$-submodule if and only if $L(\widetilde{k})=M_\pe(\widetilde{k})$ holds. 
\end{rem}
In the case where $k$ is a quartic CM-field, we can choose $\Sigma$ other than $p$-adic CM-types. 
\begin{prop}
Let $k$ be a quartic CM-field, $p$ an odd prime number that splits in $k^+$, and $\pe_+$ a $p$-adic prime of $k^+$. Suppose that $\pe_+$ splits in $k$. Then $(k, p, S_{\pe_+}(k))$ satisfies Condition~\ref{condition basic}. 
\end{prop}
\begin{proof}
Let $\pe$ be a $\pe_+$-adic prime of $k$. Then we have $S_{\pe_+}(k)=\{\pe, \overline{\pe}\}$. Observe that $\pe$-adic Leopoldt's Conjecture and $\overline{\pe}$-adic Leopoldt's Conjecture hold for $k$, because $\Z_p\otimes_\Z E(k)$ is isomorphic to $\Z_p$. Thus we obtain the assertion by Lemma~\ref{lem basic}. 
\end{proof}

\subsection{Abelian extensions over an imaginary quadratic field}\label{subsection abel imag quad}
Let $K$ be an imaginary quadratic field, $p$ a prime number that splits in $K$, and $\Pe$ a $p$-adic prime of $K$. Let $k/K$ be a finite abelian extension. In this setting, $\Pe$-adic Leopoldt's Conjecture holds for $k$ by Brumer's result \cite{Brumer1967}. 
\begin{prop}\label{prop abel imag quad}
Suppose that $\Pe$ splits completely in $k$. Then $(k, p, S_\Pe(k))$ satisfies Condition~\ref{condition basic}. 
\end{prop}
\begin{proof}
Since $\Pe$-adic Leopoldt's Conjecture holds for $k$, the $\Z_p$-rank of $X_\Pe(k)$ is $1$ by Lemma~\ref{lem two three}. Consider the $\Pe$-ramified $\Z_p$-extension $K_\infty/K$ over the imaginary quadratic field $K$. Then $K_\infty/K$ is (infinitely) ramified at $\Pe$ because of the finiteness of $[L(K):K]$ (or by Proposition~\ref{prop CM-type}). Therefore, it follows that $K_\infty k/k$ is ramified at every $\Pe$-adic prime. Since $K_\infty k/k$ is the unique $\Pe$-ramified $\Z_p$-extension over $k$, this completes the proof. 
\end{proof}

\subsection{Complex cubic fields}
\begin{prop}\label{prop cpx cubic}
Let $k$ be a cubic field that has a complex place, $p$ a prime number that splits completely in $k$, and $\Sigma$ a set consisting of two distinct $p$-adic primes of $k$. Then $(k, p, \Sigma)$ satisfies Condition~\ref{condition basic}. 
\end{prop}
\begin{proof}
Observe that $\pe$-adic Leopoldt's Conjecture holds for each $\pe\in\Sigma$, because the $\Z_p$-rank of $\Z_p\otimes_\Z E(k)$ is $1$. Thus we obtain the assertion by Lemma~\ref{lem basic}. 
\end{proof}

\subsection{Fields with a single complex place}
The following partially generalizes Proposition~\ref{prop cpx cubic}. 
\begin{prop}\label{prop single cpx}
Let $k/\Q$ be a finite extension that has a single complex place, $p$ a prime number, and $\pe_0$ a $p$-adic prime of $k$ whose degree is $1$. Put $\Sigma\coloneqq S_p(k)\setminus\{\pe_0\}$. Suppose that $p$-adic Leopoldt's Conjecture is true for $k$. Then the $\Z_p$-rank of  $X_\Sigma(k)$ is $1$ and $\Sigma$-adic Leopoldt's Conjecture holds for $k$. 
\end{prop}
\begin{proof}
Since $p$-adic Leopoldt's Conjecture is true for $k$ and $r_2(k)=1$, we have $\Gal(\widetilde{k}/k)\simeq\Z_p^2$. Let $I_0$ be the inertia group of $\pe_0$ in $\widetilde{k}/k$. It can be shown that $\rank_{\Z_p}I_0\leq 1$ holds by local class field theory because the degree of $\pe_0$ is $1$. On the other hand, we see that $I_0$ is not trivial by the existence of the cyclotomic $\Z_p$-extension. Thus $I_0$ is isomorphic to $\Z_p$. This means that the $\Z_p$-rank of  $X_\Sigma(k)$ is $1$. 

Next observe that
\[
\sum_{\pe\in\Sigma}[k_\pe:\Q_p]=[k:\Q]-1=(r_1(k)+2)-1=r_1(k)+r_2(k)
\]
holds. Hence $\Sigma$-adic Leopoldt's Conjecture holds for $k$ by Lemma~\ref{lem two three}. 
\end{proof}

\section{Theorem~\ref{thm main GGC} and its corollaries}\label{section GGC}

\subsection{Proof of Theorem~\ref{thm main GGC}}
Here we prove Theorem~\ref{thm main GGC}. Our proof of Theorem~\ref{thm main GGC} proceeds in the same manner as that of \cite[Theorem~1]{Fujii2017}, so we refer the reader to  \cite{Fujii2017} for details. 

Suppose the assumptions of Theorem~\ref{thm main GGC}. First we show the following lemma. 
\begin{lem}\label{lem free}
$U_\Sigma^{(1)}/(\Z_p\otimes_\Z E(k))$ is a free $\Z_p$-module of rank $1$. 
\end{lem}
\begin{proof}
Fix a prime $\pe\in\Sigma$. Since $(\Sigma\setminus\{\pe\})$-adic Leopoldt's Conjecture holds by Lemma~\ref{lem basic}, there is an exact commutative diagram
\[
\begin{tikzpicture}
\node (a) at (0,2) {$0$}; 
\node (b) at (2.5,2) {$\Z_p\otimes_\Z E(k)$}; 
\node (c) at (5.5,2) {$U_\Sigma^{(1)}$}; 
\node (d) at (9,2) {$U_\Sigma^{(1)}/(\Z_p\otimes_\Z E(k))$}; 
\node (e) at (12,2) {$0$}; 
\draw[->] (a) -- (b); \draw[->] (b) -- (c); \draw[->] (c) -- (d); \draw[->] (d) -- (e); 

\node (a1) at (0,0) {$0$}; 
\node (b1) at (2.5,0) {$\Z_p\otimes_\Z E(k)$}; 
\node (c1) at (5.5,0) {$U_{\Sigma\setminus\{\pe\}}^{(1)}$}; 
\node (d1) at (9,0) {$0$}; 
\draw[->] (a1) -- (b1); \draw[->] (b1) -- (c1); \draw[->] (c1) -- (d1); 

\draw[->] (b) -- (b1); \draw[->] (c) -- (c1); 
\end{tikzpicture}
\]
of $\Z_p$-modules by Condition~\ref{GGC D}. Hence there is a natural isomorphism $U_\pe^{(1)}\simeq U_\Sigma^{(1)}/(\Z_p\otimes_\Z E(k))$ by Snake Lemma. Thus we obtain the assertion because $\pe$ is of degree $1$. 
\end{proof}
By Condition~\ref{GGC C}, there is a unique $\Sigma$-ramified $\Z_p$-extension $N_\infty/k$. If $k$ is totally real, then $\Sigma$ must be $S_p(k)$ and $\widetilde{k}=N_\infty=k^\cy_\infty$, so $X(\widetilde{k})=0$ by Conditions~\ref{GGC A}, \ref{GGC C}, Lemma~\ref{lem free}, and Proposition~\ref{prop trivial}. Now suppose that $k$ has at least one complex place. If $k$ is an imaginary quadratic field, then the assertion follows from Minardi's result \cite[Proposition~3.A]{Minardi1986} (a proof of Minardi's result can also be found in \cite[\S4]{Ozaki2001}). Hence we may also suppose that $[k:\Q]$ is greater than $2$. Therefore, we can use the following lemma. 
\begin{lem}[{\cite[Theorem~25]{Maire2002}}]\label{lem Maire}
Let $k/\Q$ be a finite extension that has at least one complex place and satisfies $[k:\Q]\geq 3$. Let $p$ be a prime number, $\pe$ a $p$-adic prime of $k$ whose degree is $1$. Then $X_\pe(k')$ is finite for any finite abelian extension $k'/k$. 
\end{lem}

Since $p$ splits completely in $k$, we have $\#\Sigma=r_1(k)+r_2(k)$ by Lemma~\ref{lem two three}. So we have $\#(S_p(k)\setminus\Sigma)=r_2(k)$. Let $\qu_1, \ldots, \qu_{r_2(k)}$ be the primes in $S_p(k)\setminus\Sigma$. We put
\[
S_i\coloneqq \Sigma\cup\{\qu_1, \ldots, \qu_{i-1}\}
\]
for each $i$ with $1\leq i\leq r_2(k)+1$. Since $S_i$-adic Leopoldt's Conjecture holds by Lemma~\ref{lem Leopoldt strong}, there is a unique $S_i$-ramified $\Z_p^i$-extension $K^{(i)}/k$.  Then we obtain a tower
\[
N_\infty=K^{(1)}\subseteq K^{(2)}\subseteq \cdots \subseteq K^{(r_2(k))}\subseteq K^{(r_2(k)+1)}=\widetilde{k}
\]
of multiple $\Z_p$-extensions over $k$. For each $i$ with $1\leq i\leq r_2(k)+1$, we put $L^{(i)}\coloneqq L(K^{(i)})$, $X^{(i)}\coloneqq X(K^{(i)})$, and $\Lambda^{(i)}\coloneqq\Z_p\bbr{\Gal(K^{(i)}/k)}$. For each $i$ with $2\leq i\leq r_2(k)+1$, we put $L_i\coloneqq L^{(i)}\cap (K^{(i-1)})^\ab$. 
\begin{lem}\label{lem tot ram K1}
$K^{(1)}/k$ is totally ramified at $\Sigma$. 
\end{lem}
\begin{proof}
Let $\pe\in\Sigma$. By Conditions~\ref{GGC A} and \ref{GGC D}, we have $X_{\Sigma\setminus\{\pe\}}(k)=0$. This means that there is no non-trivial $\Sigma$-ramified abelian pro-$p$-extension that is unramified at $\pe$. Thus $K^{(1)}/k$ is totally ramified at $\pe$. 
\end{proof}
\begin{lem}
$X^{(1)}$ is trivial. 
\end{lem}
\begin{proof}
The assertion follows from the assumptions, Lemmas~\ref{lem tot ram K1}, \ref{lem free}, and Proposition~\ref{prop trivial}. 
\end{proof}
Next we show that $X^{(2)}$ is pseudo-null as a $\Lambda^{(2)}$-module. We use the following lemma. 
\begin{lem}[{\cite[Chapter~I, Lemma~4]{Perrin1984}}]\label{lem pn lift}
Let $p$ be a prime number and $d$ a positive integer. Let $G$ be a pro-$p$-group isomorphic to $\Z_p^d$ and $H$ a closed subgroup of $G$ such that $G/H$ is isomorphic to $\Z_p^{d-1}$. Let $M$ be a finitely generated $\Z_p\bbr{G}$-module. Suppose that $M_H$ is pseudo-null as a $\Z_p\bbr{G/H}$-module. Then $M$ is pseudo-null as a $\Z_p\bbr{G}$-module. 
\end{lem}
Since $\Gal(L_2/K^{(2)})\simeq X^{(2)}_{\Gal(K^{(2)}/K^{(1)})}$ holds, it suffices to show that $\Gal(L_2/K^{(2)})$ is pseudo-null as a $\Lambda^{(1)}$-module by Lemma~\ref{lem pn lift}. By Condition~\ref{GGC E}, the prime $\qu_1$ splits finitely in $K^{(1)}/k$. Take $n\geq 0$ so that $N_n$ is the decomposition field of $\qu_1$ in $K^{(1)}/k$ (recall that $K^{(1)}=N_\infty$). For each $\Qu_1\in S_{\qu_1}(K^{(1)})$, we let $I_{\Qu_1}$ denote the inertia group of $\Qu_1$ in $L_2/K^{(1)}$. Since $X^{(1)}$ is trivial, there is an equality
\[
\Gal(L_2/K^{(1)})=\sum_{\Qu_1\in S_{\qu_1}(K^{(1)})}I_{\Qu_1}
\]
of $\Lambda^{(1)}$-modules. Note that $\Gal(K^{(1)}/N_n)$ acts naturally on each $I_{\Qu_1}$. Since $L_2/K^{(2)}$ is unramified, we have $I_{\Qu_1}\cap\Gal(L_2/K^{(2)})=0$, so $I_{\Qu_1}$ injects in $\Gal(K^{(2)}/K^{(1)})$. Hence the action of $\Gal(K^{(1)}/N_n)$ on $I_{\Qu_1}$ is trivial. Therefore, we conclude that $\Gal(K^{(1)}/N_n)$ acts trivially on $\Gal(L_2/K^{(2)})$. Thus $L_2/N_n$ is an abelian extension, so $\Gal(L_2/N_n)$ is finitely generated as a $\Z_p$-module. 
\begin{prop}\label{prop rank 2}
The $\Z_p$-rank of $\Gal(L_2/N_n)$ is $2$. 
\end{prop}
\begin{proof}
One may prove the assertion by mimicking the proof of \cite[Proposition~1]{Fujii2017}, but here we give a simpler description due to Kataoka's comment. 

Let $\mathcal{M}/N_n$ be the maximal abelian pro-$p$-extension such that $\Gal(N_n/k)$ acts trivially on the inertia group of every $\pe_n\in S_\Sigma(N_n)$. Then we have $L_2\subseteq \mathcal{M}$ so it suffices to show that the $\Z_p$-rank of $\Gal(\mathcal{M}/N_n)$ is $2$. By the definition of $\mathcal{M}$, there is an exact sequence
\[
\begin{tikzpicture}
\node (a) at (0,0) {$\Z_p\otimes_\Z E(N_n)$}; 
\node (b) at (6,0) {$\displaystyle \prod_{\pe_n\in S_\Sigma(N_n)}(U_{\pe_n}^{(1)})_{\Gal(N_n/k)}\times \prod_{\qu_{1,n}\in S_{\qu_1}(N_n)}U_{\qu_{1,n}}^{(1)}$}; 
\node (c) at (12,0) {$\Gal(\mathcal{M}/N_n)$}; 
\node (d) at (14.5,0) {$0$}; 

\draw[->] (a) -- (b); \draw[->] (b) -- (c); \draw[->] (c) -- (d); 
\end{tikzpicture}
\]
of $\Z_p\br{\Gal(N_n/k)}$-modules. Since we only have to consider the $\Z_p$-ranks, we may tensor $\Q_p$ and regard the modules as representations of $\Gal(N_n/k)$. Then we can decompose $\Q_p\otimes_{\Z_p}\Gal(\mathcal{M}/N_n)=V_0\oplus V_1$, where $V_0$ corresponds to the trivial character and $V_1$ corresponds to the non-trivial characters. Here we note that $\Q_p\br{\Gal(N_n/k)}$ is semi-simple. Then, by decomposing the above exact sequence, we obtain the following exact sequences
\[
\begin{tikzpicture}
\node (a) at (0,0) {$\Q_p\otimes_\Z E(k)$}; 
\node (b) at (5,0) {$\displaystyle \left(\prod_{\pe\in\Sigma}\Q_p\otimes_{\Z_p}U_{\pe}^{(1)}\right)\times (\Q_p\otimes_{\Z_p}U_{\qu_1}^{(1)})$}; 
\node (c) at (9.5,0) {$V_0$}; 
\node (d) at (11.2,0) {$0$}; 

\draw[->] (a) -- (b); \draw[->] (b) -- (c); \draw[->] (c) -- (d); 
\end{tikzpicture}
\]
and
\[
\begin{tikzpicture}
\node (a) at (0,0) {$\Q_p\otimes_\Z E(N_n)$}; 
\node (b) at (4.5,0) {$\displaystyle \prod_{\qu_{1,n}\in S_{\qu_1}(N_n)}\Q_p\otimes_{\Z_p}U_{\qu_{1,n}}^{(1)}$}; 
\node (c) at (8.3,0) {$V_1$}; 
\node (d) at (10.5,0) {$0$}; 

\draw[->] (a) -- (b); \draw[->] (b) -- (c); \draw[->] (c) -- (d); 
\end{tikzpicture}
\]
of $\Q_p\br{\Gal(N_n/k)}$-modules. By the first sequence, we see that $\dim_{\Q_p}V_0=\rank_{\Z_p}X_{S_2}(k)=2$. By the second sequence and Lemma~\ref{lem Maire}, we see that $V_1$ is trivial. Therefore, we conclude that the $\Z_p$-rank of $\Gal(\mathcal{M}/N_n)$ is $2$, as desired. 
\end{proof}
It follows from Proposition~\ref{prop rank 2} that $\Gal(L_2/K^{(2)})$ is finite. This means that $\Gal(L_2/K^{(2)})$ is pseudo-null as a $\Lambda^{(1)}$-module. Hence $X^{(2)}$ is pseudo-null as a $\Lambda^{(2)}$-module by Lemma~\ref{lem pn lift}. 
\begin{thm}\label{thm pn step up}
For each $i$ with $2\leq i\leq r_2(k)$, if $X^{(i)}$ is pseudo-null as a $\Lambda^{(i)}$-module, then $X^{(i+1)}$ is pseudo-null as a $\Lambda^{(i+1)}$-module. 
\end{thm}
To show this theorem, one may trace the proof of \cite[Theorem~3]{Fujii2017}, but there is a little difference (see Remark~\ref{rem type}). So we carry out the proof . 

For each $\Qu_i\in S_{\qu_i}(K^{(i)})$, we let $I_{\Qu_i}$ denote the inertia group of $\Qu_i$ in $L_{i+1}/K^{(i)}$. Then there is an exact sequence
\[
\begin{tikzpicture}
\node (a) at (0,0) {$0$}; 
\node (b) at (2.5,0) {$\sum_{\Qu_i}I_{\Qu_i}$}; 
\node (c) at (6,0) {$\Gal(L_{i+1}/K^{(i)})$}; 
\node (d) at (9.5,0) {$X^{(i)}$}; 
\node (e) at (11.3,0) {$0$}; 

\draw[->] (a) -- (b); \draw[->] (b) -- (c); \draw[->] (c) -- (d); \draw[->] (d) -- (e); 
\end{tikzpicture}
\]
of $\Lambda^{(i)}$-modules. Here we note that the infinite sum $\sum_{\Qu_i}I_{\Qu_i}$ denotes the closure of the group generated by $\{I_{\Qu_i}\mid \Qu_i\in S_{\qu_i}(K^{(i)})\}$. The decomposition group of $\qu_i$ in $K^{(i)}/k$ is non-trivial since $K^{(1)}\subseteq K^{(i)}$. Let $g$ be a topological generator of the decomposition group. Since $I_{\Qu_i}\cap \Gal(L_{i+1}/K^{(i+1)})=0$ holds, we see that $I_{\Qu_i}$ is isomorphic to $\Z_p$ as a $\Z_p$-module. Hence there is a surjection $\Lambda^{(i)}/(g-1)\to \sum_{\Qu_i}I_{\Qu_i}$ of $\Lambda^{(i)}$-modules. Thus $\sum_{\Qu_i}I_{\Qu_i}$ is finitely generated and torsion as a $\Lambda^{(i)}$-module. Hence $\Gal(L_{i+1}/K^{(i)})$ is also finitely generated and torsion as a $\Lambda^{(i)}$-module. 

Let $L'_{i+1}$ be the fixed field of the maximal pseudo-null $\Lambda^{(i)}$-submodule of $\Gal(L_{i+1}/K^{(i)})$. Now we suppose that $X^{(i)}$ is pseudo-null as a $\Lambda^{(i)}$-module. Then we see that the injection $\sum_{\Qu_i}I_{\Qu_i}\to \Gal(L_{i+1}/K^{(i)})$ is a pseudo-isomorphism by the above exact sequence, so $\Gal(L'_{i+1}/K^{(i)})$ can be embedded into $\sum_{\Qu_i}I_{\Qu_i}$. Hence $g$ acts trivially on $\Gal(L'_{i+1}/K^{(i)})$. If we let $E\coloneqq (K^{(i)})^{\tgen{g}}$ be the decomposition field of $\qu_i$ in $K^{(i)}/k$, then $L'_{i+1}/E$ is an abelian extension. The following lemma corresponds to \cite[Lemma~10]{Fujii2017}. 
\begin{lem}\label{lem ramified}
Let $E'$ be the subfield of $E$ such that $E'/k$ is a $\Z_p^{i-1}$-extension. Then $E'/k$ is ramified at every prime in $S_i$. 
\end{lem}
\begin{proof}
To derive a contradiction, suppose that $E'/k$ is unramified at some prime $\pe\in S_i$. 

First we consider the case where $\pe\in \Sigma$. By Lemma~\ref{lem basic}, we see that $(\Sigma\setminus\{\pe\})$-adic Leopoldt's Conjecture holds. Then we see that $E'$ is the unique $((\Sigma\setminus\{\pe\})\cup\{\qu_1, \ldots, \qu_{i-1}\})$-ramified $\Z_p^{i-1}$-extension over $k$. Consider the $((\Sigma\setminus\{\pe\})\cup\{\qu_1\})$-ramified $\Z_p$-extension $K/k$. By Condition~\ref{GGC F}, the prime $\qu_i$ splits finitely in $K$. However, since $K$ is contained in $E$, this is a contradiction. 

Next we consider the case where $\pe\in S_i\setminus\Sigma$. Then we see that $E'$ is the unique $S_i\setminus\{\pe\}$-ramified $\Z_p^{i-1}$-extension over $k$. By Condition~\ref{GGC E}, the prime $\qu_i$ splits finitely in $K^{(1)}$. However, since $K^{(1)}$ is contained in $E$, this is a contradiction. 
\end{proof}
\begin{rem}\label{rem type}
In the proof of \cite[Lemma~10]{Fujii2017}, Fujii divided the argument into two cases, according to whether $S_i\setminus\{\pe\}$ is ``type~1'' or ``type~2''. The concept of ``type'' in \cite{Fujii2017} is specific to CM-fields. Our contribution here is to use the validity of $(\Sigma\setminus\{\pe\})$-adic Leopoldt's Conjecture, what we checked in Lemma~\ref{lem basic}, instead of the concept of ``type''. See \S\ref{subsubsection earlier Fujii} for the relation between these approaches. 
\end{rem}
The remaining part of the proof of Theorem~\ref{thm pn step up} proceeds in the same way as in \cite{Fujii2017}, so we only show the outline. 

By Lemma~\ref{lem ramified}, the extension $K^{(i)}/E$ is $\qu_i$-ramified. Hence $L'_{i+1}/E$ is a $\qu_i$-ramified abelian pro-$p$-extension. Let $F$ be an intermediate field of $E/k$ such that $E/F$ is a $\Z_p^{i-1}$-extension. Note that $F/k$ is a finite extension. 
\begin{prop}
$\Gal(L'_{i+1}/E)$ is finitely generated and torsion as a $\Z_p\bbr{\Gal(E/F)}$-module. 
\end{prop}
\begin{proof}
Since $L'_{i+1}$ is contained in $M_{\qu_i}(E)$, it suffices to show that $X_{\qu_i}(E)$ is finitely generated and torsion as a $\Z_p\bbr{\Gal(E/F)}$-module. 

Since $p$ splits completely in $k$, by \cite[Lemma~5]{Fujii2017}, we can take a $\Z_p$-extension $F_\infty/F$ contained in $E$ such that $E/F_\infty$ is unramified. Since $X_{\qu_i}(F_n)$ is finite for every $n\geq 0$ by Lemma~\ref{lem Maire}, we can show that $X_{\qu_i}(F_\infty)$ is finitely generated and torsion as a $\Z_p\bbr{\Gal(F_\infty/F)}$-module by a similar method as in the proof of Iwasawa's class number formula. Since $E/F_\infty$ is unramified, there is an exact sequence
\[
\begin{tikzpicture}
\node (a) at (0,0) {$\Homo_2(\Gal(E/F_\infty), \Z_p)$}; 
\node (b) at (4.2,0) {$X_{\qu_i}(E)_{\Gal(E/F_\infty)}$}; 
\node (c) at (7.6,0) {$X_{\qu_i}(F_\infty)$}; 

\draw[->] (a) -- (b); \draw[->] (b) -- (c); 
\end{tikzpicture}
\]
of $\Z_p\bbr{\Gal(F_\infty/F)}$-modules, which is a part of the five-term exact sequence in group homology. Hence $X_{\qu_i}(E)_{\Gal(E/F_\infty)}$ is finitely generated and torsion as a $\Z_p\bbr{\Gal(F_\infty/F)}$-module, so $X_{\qu_i}(E)$ is finitely generated as a $\Z_p\bbr{\Gal(E/F)}$-module by Nakayama's Lemma and torsion as a $\Z_p\bbr{\Gal(E/F)}$-module by Cayley--Hamilton Theorem. 
\end{proof}
Since $\Gal(L'_{i+1}/E)$ is finitely generated and torsion as a $\Z_p\bbr{\Gal(E/F)}$-module and $g$ acts trivially on $\Gal(L'_{i+1}/K^{(i)})$, it follows that $\Gal(L'_{i+1}/K^{(i)})$ is pseudo-null as a $\Z_p\bbr{\Gal(K^{(i)}/F)}$-module. Since $\Gal(K^{(i)}/F)$ is a finite index subgroup of $\Gal(K^{(i)}/k)$, the module $\Gal(L'_{i+1}/K^{(i)})$ is pseudo-null as a $\Z_p\bbr{\Gal(K^{(i)}/k)}$-module as well. Therefore, by the definition of $L'_{i+1}$, we see that $\Gal(L'_{i+1}/K^{(i)})$ is trivial, so $\Gal(L_{i+1}/K^{(i)})$ is pseudo-null as a $\Z_p\bbr{\Gal(K^{(i)}/k)}$-module. Then $\Gal(L_{i+1}/K^{(i+1)})$ is also pseudo-null as a $\Lambda^{(i)}$-module. Here we remark that there is a minor error in the argument in \cite{Fujii2017} that corresponds to the above argument. A module of the form $\Gal(L'_{i+1}/K^{(i+1)})$ appears in \cite{Fujii2017}, but such a module does not exist because $L'_{i+1}=K^{(i)}$. Since there is an isomorphism
\[
X^{(i+1)}_{\Gal(K^{(i+1)}/K^{(i)})}\simeq \Gal(L_{i+1}/K^{(i+1)})
\]
of $\Lambda^{(i)}$-modules, we conclude that $X^{(i+1)}$ is pseudo-null as a $\Lambda^{(i+1)}$-module by Lemma~\ref{lem pn lift}, as desired. This completes the proof of Theorem~\ref{thm pn step up}. 

We showed that $X^{(2)}$ is pseudo-null as a $\Lambda^{(2)}$-module. Therefore, we conclude that $X^{(r_2(k)+1)}$ is pseudo-null as a $\Lambda^{(r_2(k)+1)}$-module by using Theorem~\ref{thm pn step up} repeatedly. This completes the proof of Theorem~\ref{thm main GGC}.

\subsection{Proof of Corollaries~\ref{cor GGC abel imag quad} and \ref{cor GGC single complex}}
\begin{proof}[Proof of Corollary~\ref{cor GGC abel imag quad}]
It is sufficient to show the conditions in Theorem~\ref{thm main GGC}. Conditions~\ref{GGC A} and \ref{GGC F} are assumed in Corollary~\ref{cor GGC abel imag quad}. Conditions~\ref{GGC B} and \ref{GGC C} hold as we mentioned in \S\ref{subsection abel imag quad}. Note that $X_\Sigma(k)$ is isomorphic to $\Z_p$ by Condition~\ref{b}, so $M_\Sigma(k)/k$ is the $\Sigma$-ramified $\Z_p$-extension over $k$. By Condition~\ref{c}, the unique $\Pe$-ramified $\Z_p$-extension $N_\infty/K$ is totally ramified at $\Pe$, so $M_\Sigma(k)/k$ is totally ramified at $S_\Pe(k)$ because $M_\Sigma(k)=N_\infty k$ and $p$ splits completely in $k$. Hence Condition~\ref{GGC D} is satisfied. Since $\overline{\Pe}$ splits finitely in $N_\infty/K$ by \cite[Lemma~3.1]{Minardi1986}, it follows that Condition~\ref{GGC E} holds. This completes the proof. 
\end{proof}
\begin{proof}[Proof of Corollary~\ref{cor GGC single complex}]
It is sufficient to show the conditions in Theorem~\ref{thm main GGC}. Conditions~\ref{GGC A}, \ref{GGC D}, and \ref{GGC E} are assumed in Corollary~\ref{cor GGC single complex}. 

Now we claim that $p$-adic Leopoldt's Conjecture for $k$ follows from Condition~\ref{GGC c}. Take a prime $\pe\in \Sigma$. Then $U_{\Sigma\setminus\{\pe\}}^{(1)}/(\Z_p\otimes_\Z E(k))$ is trivial by Condition~\ref{GGC c}. By comparing the $\Z_p$-ranks of $U_{\Sigma\setminus\{\pe\}}^{(1)}$ and $\Z_p\otimes_\Z E(k)$, we see that $(\Sigma\setminus\{\pe\})$-adic Leopoldt's Conjecture holds, so $p$-adic Leopoldt's Conjecture also holds, as desired. 

Conditions~\ref{GGC B} and \ref{GGC C} hold by Proposition~\ref{prop single cpx}. Condition~\ref{GGC F} trivially holds because there is no pair of different primes in $S_p(k)\setminus\Sigma$. This completes the proof. 

Here we note that the proof essentially ends at Proposition~\ref{prop rank 2} because $K^{(2)}$ coincides with $\widetilde{k}$ in this case. 
\end{proof}

\subsection{Relation with earlier works}\label{subsection earlier}

\subsubsection{Relation with Theorem~\ref{thm Fujii GGC}}\label{subsubsection earlier Fujii}
Here we show that Theorem~\ref{thm main GGC} recovers Theorem~\ref{thm Fujii GGC}. 

Now suppose the assumptions in Theorem~\ref{thm Fujii GGC}. Let $\Sigma$ be a $p$-adic CM-type  of $k$. It is sufficient to show the conditions in Theorem~\ref{thm main GGC}. We may also assume that $k$ is not an imaginary quadratic field by Minardi's result \cite[Proposition~3.A]{Minardi1986}. 

Condition~\ref{GGC A} is assumed in Theorem~\ref{thm Fujii GGC}. Note that $p$-adic Leopoldt's Conjecture is true for $k$, as we mentioned after the statement of Theorem~\ref{thm Fujii GGC}. Thus $\Sigma$-adic Leopoldt's Conjecture holds for $k$ by Lemma~\ref{lem p and Sigma}. This shows Condition~\ref{GGC B}. Condition~\ref{GGC C} is shown in Proposition~\ref{prop CM unique} and Condition~\ref{GGC D} follows from \cite[Lemma~9]{Fujii2017}. 

To show Conditions~\ref{GGC E} and \ref{GGC F}, we use the following lemma. 
\begin{lem}[{\cite[Lemma~3]{Fujii2017}}]\label{lem type split}
Let $S$ be a $p$-adic CM-type or $S=\{\pe,\overline{\pe}\}$ for some $p$-adic prime $\pe$. Let $\mathfrak{l}$ be a prime of $k$ that is not contained in $S$ and splits in $k/k^+$. If $S=\{\pe, \overline{\pe}\}$, suppose that $k$ is not an imaginary quadratic field. Then $\mathfrak{l}$ splits finitely in $M_S(k)/k$. 
\end{lem}
Condition~\ref{GGC E} follows directly from this lemma. 

To show Condition~\ref{GGC F}, let $\pe_0, \pe_1, \pe_2\in\Sigma$ such that $\pe_1\neq \pe_2$ and put $S=(\Sigma\setminus\{\pe_0\})\cup\{\overline{\pe_1}\}$. It is sufficient to show that $\overline{\pe_2}$ splits finitely in $M_S(k)/k$. First consider the case where $\pe_0=\pe_1$. Then we see that $S$ is a $p$-adic CM-type, so the assertion follows from Lemma~\ref{lem type split}. Next suppose that $\pe_0\neq \pe_1$. Then, if we put $S'\coloneqq\{\pe_1, \overline{\pe_1}\}$, we see that $S'$ is a subset of $S$. The $\Z_p$-rank of $X_{S'}(k)$ is $1$ by \cite[Lemma~2~(2)]{Fujii2017}. Since the $\Z_p$-rank of $X_S(k)$ is also $1$, it follows that $M_S(k)/M_{S'}(k)$ is a finite extension. Therefore, since $\pe_2$ splits finitely in $M_{S'}(k)/k$ by Lemma~\ref{lem type split}, we conclude that $\pe_2$ splits finitely in $M_S(k)$ as well. This completes the proof.

\subsubsection{Relation with Theorem~\ref{thm Kataoka}}\label{subsubsection earlier Kataoka}
Here we show that Theorem~\ref{thm main GGC} recovers Theorem~\ref{thm Kataoka}. It is sufficient to deduce Theorem~\ref{thm Kataoka} from Corollary~\ref{cor GGC single complex}. 

Now suppose the assumptions of Theorem~\ref{thm Kataoka}. Since $p^2$ does not divide $\varepsilon^{p-1}-1$, it follows that $\varepsilon^{p-1}\not\equiv 1\pmod{\pe_1^2}$ holds for some $\pe_1\in S_p(k)$ because $p^2=\prod_{\pe\in S_p(k)}\pe^2$. Then $U_{\pe_1}^{(1)}/(\Z_p\otimes_\Z E(k))$ is trivial. By \cite[Lemma~3.9]{Kataoka2017}, there is another $p$-adic prime $\pe_2\in S_p(k)\setminus\{\pe_1\}$ such that $U_{\pe_2}^{(1)}/(\Z_p\otimes_\Z E(k))$ is trivial. Let $\pe_0\in S_p(k)\setminus\{\pe_1, \pe_2\}$ be the other $p$-adic prime and put $\Sigma\coloneqq S_p(k)\setminus\{\pe_0\}=\{\pe_1, \pe_2\}$. It is sufficient to show the assumptions of Corollary~\ref{cor GGC single complex} in this setting. 

Condition~\ref{GGC a} is assumed in Theorem~\ref{thm Kataoka}. Condition~\ref{GGC c} holds by the choice of $\pe_0$, as we mentioned above. Since $p$ splits finitely in $\widetilde{k}/\Q$, it follows that $\pe_0$ splits finitely in $M_\Sigma(k)/k$. This shows Condition~\ref{GGC d} and we complete the proof.

\subsection{Conditional results}\label{subsection conditional}
\subsubsection{Weak $p$-adic Schanuel Conjecture}
Here we show some conditional results depending on weak $p$-adic Schanuel Conjecture. For a prime number $p$, we fix an algebraic closure $\overline{\Q_p}$ of $\Q_p$ and let $\log_p\colon \overline{\Q_p}^\times\to\overline{\Q_p}$ be the $p$-adic logarithm normalized by $\log_pp=0$. 
\begin{conjecture}[weak $p$-adic Schanuel Conjecture, {\cite[Conjecture~3.10]{CM2009}}]\label{conjecture wpSC}
Let $\alpha_1, \ldots, \alpha_n\in\overline{\Q_p}$ be non-zero algebraic numbers. If 
\[
\log_p\alpha_1, \ldots, \log_p\alpha_n
\]
are linearly independent over $\Q$, then they are algebraically independent over $\Q$. 
\end{conjecture}
\subsubsection{Construction of $\Sigma$}\label{subsection construction}
The aim of the present subsection is to show Proposition~\ref{prop existence of Sigma} below. Here we construct a set $\Sigma$. 

Let $k/\Q$ be a finite extension and $p$ an odd prime number that splits completely in $k$. For each $\pe\in S_p(k)$, we let $\iota_\pe\colon k\to k_\pe\simeq \Q_p$ denote the canonical embedding. Put $r\coloneqq r_1(k)+r_2(k)-1$. 

Let $k'$ be the Galois closure of $k/\Q$ and fix an embedding $j_p\colon k'\to \Q_p$. Put $G\coloneqq\Gal(k'/\Q)$ and $H\coloneqq\Gal(k'/k)$. For each $\pe\in S_p(k)$, take an element $\sigma_\pe\in G$ that satisfies $\iota_\pe=j_p\circ \sigma_\pe{\mid}_k$. Then $\{\sigma_\pe\mid \pe\in S_p(k)\}$ is a complete set of representatives of $G/H$. 

Fix an embedding $j_\infty\colon k'\to \C$ and let $\{\rho_1, \ldots, \rho_{r+1}\}$ be a complete set of representatives of the embeddings $k\to\C$. For each $i$, take an element $\sigma_i\in G$ that satisfies $\rho_i=j_\infty\circ \sigma_i{\mid}_k$, and let $\pe_i$ be the $p$-adic prime such that $\sigma_i{\mid}_k=\sigma_{\pe_i}{\mid}_k$. Now we put $\Sigma\coloneqq\{\pe_1, \ldots, \pe_{r+1}\}$. 
\begin{prop}\label{prop Schanuel Conditions}\label{prop existence of Sigma}
Assume that Conjecture~\ref{conjecture wpSC} is true. Then the set $\Sigma$ constructed above satisfies Conditions~\ref{GGC B}, \ref{GGC C}, \ref{GGC E}, and \ref{GGC F} in Theorem~\ref{thm main GGC}. 
\end{prop}
We obtain the following conditional result by combining Theorem~\ref{thm main GGC} and Proposition~\ref{prop existence of Sigma}
\begin{thm}\label{thm conditional GGC}
Let $k/\Q$ be a finite extension and $p$ an odd prime number that splits completely in $k$. Assume that Conjecture~\ref{conjecture wpSC} is true. Suppose that $p$ does not divide the  class number of $k$ and the set $\Sigma$ constructed above satisfies $U^{(1)}_{\Sigma\setminus\{\pe\}}/(\Z_p\otimes_\Z E(k))=0$ for every $\pe\in\Sigma$. Then Conjecture~\ref{conjecture GGC} is true for $k$ and $p$. 
\end{thm}
\subsubsection{Proof of Proposition~\ref{prop existence of Sigma}}
Now we introduce a weak version of Condition~\ref{GGC D} in Theorem~\ref{thm main GGC}.  
\begin{enumerate}
\item[(D')] $U_{\Sigma\setminus\{\pe\}}^{(1)}/(\Z_p\otimes_\Z E(k))$ is finite for every $\pe\in\Sigma$. 
\end{enumerate}
The strategy for Proposition~\ref{prop Schanuel Conditions} is as follows. We show in Lemma~\ref{lem BCD'} that the set $\Sigma$ above satisfies Conditions~\ref{GGC B}, \ref{GGC C}, and (D'). Then we show in Proposition~\ref{prop BCD' to EF} that any set $\Sigma$ satisfying Conditions~\ref{GGC B}, \ref{GGC C}, and (D') automatically satisfies Conditions~\ref{GGC E} and \ref{GGC F}. We use Conjecture~\ref{conjecture wpSC} in both steps. 

Fix fundamental units $\ve_1, \ldots, \ve_r$ of $k$. For any set $S$ of $p$-adic primes of $k$, we define a matrix
\[
A_S\coloneqq \left( \log_p \iota_{\pe}(\ve_j)  \right)_{\pe\in S, 1\leq j\leq r}
\]
with coefficients in $\Q_p$. Since $p$ is odd, there is an isomorphism $U_S^{(1)}\simeq \bigoplus_{\pe\in S}p\Z_p$ induced by $\log_p$. Note that the matrix $A_S$ represents the morphism $\Q_p\otimes E(k)\to \Q_p^S$ induced by $\Z_p\otimes E(k)\to U_S^{1}\simeq \bigoplus_{\pe\in S}p\Z_p$. So $S$-adic Leopoldt's Conjecture is equivalent to the assertion that the rank of $A_S$ over $\Q_p$ is $r$. 

\begin{lem}\label{lem BCD'}
Assume that Conjecture~\ref{conjecture wpSC} is true. Then the set $\Sigma$ constructed in \S\ref{subsection construction} satisfies Conditions~\ref{GGC B}, \ref{GGC C} in Theorem~\ref{thm main GGC}, and (D'). 
\end{lem}
\begin{proof}
The method is standard and not new. For example, see the proof of \cite[Lemma~1.2.1]{HT1994}. 

Let $z_1, \ldots, z_d\in\{\log_p j_p(\sigma_i(\ve_j))\mid 1\leq i\leq r+1, 1\leq j\leq r\}$ be a $\Q$-basis of $\sum_{1\leq i\leq r+1, 1\leq j\leq r}\Q\log_pj_p(\sigma_i(\ve_j))$. Then write as $z_t=\log_pj_p(\sigma_{i_t}(\ve_{j_t}))$ and put $y_t\coloneqq \log\abs{j_\infty(\sigma_{i_t}(\ve_{j_t}))}$ for each $t$ ($1\leq t\leq d$). 

We write $\log_p j_p(\sigma_i(\ve_j))=\sum_{t=1}^d a_{ij}^{(t)}z_t$ with $a_{ij}^{(t)}\in \Q$, then we see that $\log\abs{j_\infty(\sigma_i(\ve_j))}=\sum_{t=1}^d a_{ij}^{(t)}y_t$ also holds as follows. Take $a\in \Z\setminus\{0\}$ so that $b_{ij}^{(t)}\coloneqq aa_{ij}^{(t)}\in\Z$ for all $i$, $j$, and $t$. Then we see that 
\[
\log_p j_p\left( \sigma_i(\ve_j)^a\prod_{t=1}^d\sigma_{i_t}(\ve_{j_t})^{-b_{ij}^{(t)}} \right)=0
\]
holds, so $\sigma_i(\ve_j)^a\prod_{t=1}^d\sigma_{i_t}(\ve_{j_t})^{-b_{ij}^{(t)}}$ is a root of unity. Hence we see that
\[
\log\left\lvert j_\infty\left( \sigma_i(\ve_j)^a\prod_{t=1}^d\sigma_{i_t}(\ve_{j_t})^{-b_{ij}^{(t)}}\right) \right\rvert=0
\]
holds, thus the assertion follows. 

For each $i_0$ ($1\leq i_0\leq r+1$), let $P_{i_0}(Z_1, \ldots, Z_d)\in\Q\br{Z_1, \ldots, Z_d}$ be the polynomial defined by
\[
P_{i_0}(Z_1, \ldots, Z_d)\coloneqq\det\left( \sum_{t=1}^d a_{ij}^{(t)}Z_t \right)_{1\leq i\leq r+1, i\neq i_0, 1\leq j\leq r}. 
\] 
Then $P_{i_0}(z_1, \ldots, z_d) =\det A_{\Sigma\setminus\{\pe_{i_0}\}}$. By the above observation, we have
\[
P_{i_0}(y_1, \ldots, y_d)=\det\left( \log \abs{\rho_i(\ve_j)} \right)_{1\leq i\leq r+1, i\neq i_0, 1\leq j\leq r}, 
\]
so its absolute value coincides with the regulator of $k$ up to a power of $2$. Thus $P_{i_0}(y_1, \ldots, y_d)\neq 0$ by Dirichlet unit theorem, so the polynomial $P_{i_0}(Z_1, \ldots, Z_d)$ is non-zero. Now, Conjecture~\ref{conjecture wpSC} asserts that $z_1, \ldots, z_d$ are algebraically independent over $\Q$. Therefore, we conclude that $\det A_{\Sigma\setminus\{\pe_{i_0}\}}=P_{i_0}(z_1, \ldots, z_d)\neq 0$. This shows Condition~(D'). 

Since $\#\Sigma=r+1$, by comparing the $\Z_p$-ranks of $U^{(1)}_{\Sigma\setminus\{\pe_{i_0}\}}$ and $\Z_p\otimes_{\Z}E(k)$, Condition~(D') implies that $(\Sigma\setminus\{\pe_{i_0}\})$-adic Leopoldt's Conjecture holds, so Condition~\ref{GGC B} holds. Then Condition~\ref{GGC C} holds by Lemma~\ref{lem two three}. 
\end{proof}
\begin{prop}\label{prop BCD' to EF}
Assume that Conjecture~\ref{conjecture wpSC} is true. Let $\Sigma$ be a set of $p$-adic primes of $k$ that satisfies Conditions~\ref{GGC B}, \ref{GGC C} in Theorem~\ref{thm main GGC}, and (D'). Then Conditions~\ref{GGC E} and \ref{GGC F} are also satisfied. 
\end{prop}
We prove the proposition in a slightly generalized form. 
\begin{prop}\label{prop split finitely}
Assume that Conjecture~\ref{conjecture wpSC} is true. Let $S$ be a set of $p$-adic primes of $k$ that satisfies $\#S=r$. Suppose that $S$-adic Leopoldt's Conjecture holds for $k$. Then $\qu'$ splits finitely in $M_{S\cup\{\qu\}}(k)/k$, for every $\qu, \qu'\in S_p(k)\setminus S$ with $\qu\neq \qu'$. 
\end{prop}
\begin{proof}
Our method is a modification of Fujii's method in \cite[Lemma~3]{Fujii2017}, which is based on Itoh's method in \cite[Lemma~2.2]{Itoh2011}. 

Put $S'\coloneqq S\cup\{\qu\}$ and take an element $\alpha\in k^\times$ that satisfies $(\qu')^{h_k(p-1)}=(\alpha)$ and $\alpha\equiv 1\pmod{\pe}$ for all $\pe\in S'$.
Let $z_1, \ldots, z_d\in\{\log_p\iota_\pe(\ve_j)\mid \pe\in S', 1\leq j\leq r\}$ be a $\Q$-basis of $\sum_{\pe\in S', 1\leq j\leq r}\Q\log_p\iota_\pe(\ve_j)$. 

Now we claim that $z_1, \ldots, z_d, (\log_p\iota_\pe(\alpha))_{\pe\in S'}$ are linearly independent over $\Q$. Write as $z_t=\log_p\iota_{\pe_t}(\ve_{j_t})$, and suppose that
\[
\sum_{t=1}^d a_t z_t+\sum_{\pe\in S'}a_\pe \log_p\iota_\pe(\alpha)=0
\]
for some $a_t, a_\pe\in \Q$. By clearing the denominators, we may suppose that $a_t, a_\pe\in \Z$. Then
\[
\log_pj_p\left( \prod_{t=1}^d \sigma_{\pe_t}(\ve_{j_t})^{a_t} \prod_{\pe\in S'}\sigma_\pe(\alpha)^{a_\pe} \right)=0. 
\]
Note that $j_p(\sigma_\pe(\alpha))$ is a $p$-adic unit for every $\pe\in S'$. Thus it follows that $\prod_{t=1}^d \sigma_{\pe_t}(\ve_{j_t})^{a_t} \prod_{\pe\in S'}\sigma_\pe(\alpha)^{a_\pe}$ is a root of unity. Considering the valuation at a $\sigma_{\pe}(\qu')$-adic prime of $k'$, we obtain $a_\pe=0$ for all $\pe\in S'$. Then $\sum_{t=1}^d a_tz_t=0$, so $a_t=0$ for all $t$. 

Now, Conjecture~\ref{conjecture wpSC} asserts that $z_1, \ldots, z_d, (\log_p\iota_\pe(\alpha))_{\pe\in S'}$ are algebraically independent over $\Q$. Consider a matrix
\[
A\coloneqq\left( (\log_p\iota_{\pe}(\ve_j))_{\pe\in S', 1\leq j\leq r}\mid (\log_p\iota_\pe(\alpha))_{\pe\in S'} \right), 
\]
and we claim that $\det A$ is non-zero. 
Write each component $\log_p\iota_\pe(\ve_j)$ as a $\Q$-linear combination of $z_1, \ldots, z_d$, replace each $z_t$ by an indeterminate $Z_t$, and each component $\log_p\iota_{\pe}(\alpha)$ by an indeterminate $Y_\pe$, then we obtain a polynomial $P\in\Q\br{Z_1, \ldots, Z_d, (Y_\pe)_{\pe\in S'}}$ such that $P(z_1, \ldots, z_d, (\log_p\iota_\pe(\alpha))_{\pe\in S'})=\det A$. If we expand $P(z_1, \ldots, z_d, (Y_\pe)_{\pe\in S'})$ along the rightmost column, the coefficient of $Y_\qu$ is $\det A_S$ up to sign, which is non-zero by $S$-adic Leopoldt's Conjecture. Thus the polynomial $P$ is non-zero. Therefore, $\det A$ is non-zero, as well. 

Lastly, we show that $\qu'$ splits finitely in $M_{S'}(k)/k$ by class field theory. By operation $\Q_p\otimes_{\Z_p}(-)$ to the exact sequence
\[
\begin{tikzpicture}
\node (b) at (0,0) {$\Z_p\otimes_\Z E(k)$}; 
\node (c) at (3,0) {$U_{S'}^{(1)}$}; 
\node (d) at (5.7,0) {$X_{S'}(k)$}; 
\node (e) at (8.4,0) {$X(k)$}; 
\node (f) at (10.2,0) {$0$};
\draw[->] (b) -- (c); \draw[->] (c) -- (d); \draw[->] (d) -- (e); \draw[->] (e) -- (f);
\end{tikzpicture}
\]
of $\Z_p$-modules, we obtain an isomorphism
\[
\Q_p\otimes_{\Z_p}X_{S'}(k)\simeq  \Q_p^{S'}/ \left(\sum_{1\leq j\leq r}\Q_p(\log_p\iota_\pe(\ve_j))_{\pe\in S'} \right)
\]
of $\Q_p$-vector spaces. Since $S$-adic Leopoldt's Conjecture holds, the $\Z_p$-rank of $X_{S'}(k)$ is $1$. 

Since $M_{S'}(k)/k$ is unramified at $\qu'$, the corresponding decomposition group is generated by the Frobenius element $\varphi_{\qu'}\in X_{S'}(k)$. Since the $\Z_p$-rank of $X_{S'}(k)$ is $1$, the decomposition group is open if and only if $1\otimes\varphi_{\qu'}$ is non-trivial in $\Q_p\otimes_{\Z_p}X_{S'}(k)$. Since the principal id\`{e}le $(\alpha)_v$ and the id\`{e}le $(\alpha)_{v\not\in S'\cup\{\qu'\}}$ (the components at $v\in S'\cup\{\qu'\}$ are $1$) map to the identity in $X_{S'}(k)$ by reciprocity, the id\`{e}le $(\alpha)_{v\in S'}\in U^{(1)}_{S'}$ maps to $h_k(p-1)\varphi_{\qu'}$ up to sign. Therefore, it suffices to show that $(\log_p\iota_\pe(\alpha))_{\pe\in S'}$ is not in $\sum_{1\leq j\leq r}\Q_p(\log_p\iota_\pe(\ve_j))_{\pe\in S'}$, which is equivalent to $\det A\neq 0$ because the first $r$ columns of $A$ are linearly independent. As we confirmed $\det A\neq 0$ above, this completes the proof. 
\end{proof}
\begin{proof}[Proof of Proposition~\ref{prop BCD' to EF}]
Recall that $\#\Sigma=r+1$ by Lemma~\ref{lem two three}. Then $(\Sigma\setminus\{\pe\})$-adic Leopoldt's Conjecture holds for every $\pe\in\Sigma$ by Condition~(D'). 

To show Condition~\ref{GGC E}, take any $\pe_0\in\Sigma$ and apply Proposition~\ref{prop split finitely} to $S=\Sigma\setminus\{\pe_0\}$ and $\qu=\pe_0$. To show Condition~\ref{GGC F}, apply Proposition~\ref{prop split finitely} to $S=\Sigma\setminus\{\pe\}$. 
\end{proof}
\begin{proof}[Proof of Proposition~\ref{prop existence of Sigma}]
The set $\Sigma$ satisfies Conditions~\ref{GGC B}, \ref{GGC C}, and (D') by Lemma~\ref{lem BCD'}. Then $\Sigma$ satisfies Conditions~\ref{GGC E} and \ref{GGC F} by Proposition~\ref{prop BCD' to EF}. This completes the proof. 
\end{proof}

\subsection{Applications of Theorem~\ref{thm main GGC} and Corollary~\ref{cor GGC abel imag quad}}\label{subsection applications}

\subsubsection{An application to non-abelian Iwasawa theory}\label{subsubsection non-abelian}
As mentioned in \cite[\S1]{Fujii2017}, there is an application of Conjecture~\ref{conjecture GGC} to non-abelian Iwasawa theory in the sense of Ozaki \cite{Ozaki2007}. 

Let $k/\Q$ be a finite extension and $p$ a prime number. We let $\mathcal{G}(k^\cy_\infty)$ denote the Galois group of the maximal unramified pro-$p$-extension of $k^\cy_\infty$. The following is conjectured by Ozaki (see \cite[p.531, Conjecture]{Fujii2022}). 
\begin{conjecture}[Non-freeness Conjecture]\label{conjecture non-free}
Let $k/\Q$ be a finite extension and $p$ a prime number. Then $\mathcal{G}(k^\cy_\infty)$ is not a non-abelian free pro-$p$-group. 
\end{conjecture}
Fujii showed that Conjecture~\ref{conjecture GGC} implies Conjecture~\ref{conjecture non-free} in the case where $p$ splits completely in $k$. 
\begin{thm}[{\cite[Theorem~1.2]{Fujii2011}}]\label{thm GGC non-free}
Let $k/\Q$ be a finite extension and $p$ a prime number that splits completely in $k$. Suppose that Conjecture~\ref{conjecture GGC} is true for $k$ and $p$. Then Conjecture~\ref{conjecture non-free} is also true for $k$ and $p$. 
\end{thm}
We obtain the following by combining Theorems~\ref{thm GGC non-free} and \ref{thm main GGC}. 
\begin{cor}
Suppose the assumptions of Theorem~\ref{thm main GGC}. Then Conjecture~\ref{conjecture non-free} is true for $k$ and $p$. 
\end{cor}

\subsubsection{An application to Iwasawa invariants of $\Z_p$-extensions}\label{subsubsection Iwasawa invariants}
Here we consider an application of Corollary~\ref{cor GGC abel imag quad} to the study of the behavior of Iwasawa invariants of $\Z_p$-extensions over a fixed finite extension $k/\Q$ and a prime number $p$. 

First we introduce a conjecture. Let $k$ be a finite extension and $p$ a prime number. We let $X'(k^\cy_\infty)$ denote the Galois group of the maximal unramified abelian pro-$p$-extension of $k^\cy_\infty$ in which every $p$-adic prime of $k^\cy_\infty$ splits completely. 
\begin{conjecture}[Generalized Gross Conjecture]\label{conjecture Gross}
Let $k/\Q$ be a finite extension and $p$ a prime number. Then $X'(k^\cy_\infty)_{\Gal(k^\cy_\infty/k)}$ is finite. 
\end{conjecture}
This conjecture is also known as Gross--Kuz'min Conjecture. For the above formulation, see \cite[Theorem~5]{Jaulent2017} and its references. 

Conjecture~\ref{conjecture Gross} is classically known to hold for abelian number fields (\cite{Greenberg1973Inv}) and recently proved for abelian extensions over imaginary quadratic fields  by Maksoud. 
\begin{thm}[{\cite[Corollary~1.7]{Maksoud2023}}]\label{thm Maksoud}
Let $k$ be a finite abelian extension over an imaginary quadratic field and $p$ a prime number. Then  Conjecture~\ref{conjecture Gross} is true for $k$ and $p$. 
\end{thm}
Next we introduce Kataoka's result in \cite{Kataoka2017JNT}. Let $k/\Q$ be a finite extension and $p$ a prime number. We let $\mathcal{E}(k)$ denote the set of all $\Z_p$-extensions of $k$. Greenberg introduced a profinite topology on $\mathcal{E}(k)$ in \cite{Greenberg1973}. If we let $\mathcal{E}_{\mathrm{ns}}(k)$ denote the subset of $\mathcal{E}(k)$ that consists of all $\Z_p$-extensions over $k$ in which every $p$-adic prime of $k$ does not split. For a $\Z_p$-extension $k_\infty/k$, we let $\mu(k_\infty/k)$ and $\lambda(k_\infty/k)$ denote the Iwasawa $\mu$- and $\lambda$-invariants of $k_\infty/k$, respectively. With this notation, Kataoka showed the following. 
\begin{thm}[{\cite[Theorem~5.3]{Kataoka2017JNT}}]\label{thm Kataoka conseq}
Let $k/\Q$ be a finite extension and $p$ a prime number that splits completely in $k$. Suppose that $p$-adic Leopoldt's Conjecture is true for $k$. Suppose also that Conjectures~\ref{conjecture GGC} and \ref{conjecture Gross} are true for $k$ and $p$. Then the set
\[
\{k_\infty\in\mathcal{E}_{\mathrm{ns}}(k)\mid \mu(k_\infty/k)=0, \lambda(k_\infty/k)=r_2(k)\}
\]
contains an open dense subset of $\mathcal{E}_{\mathrm{ns}}(k)$. 
\end{thm}
\begin{proof}
With the notation in \cite[Theorem~5.3]{Kataoka2017JNT}, the vanishing of $s'(k)$ follows from the validity of Conjecture~\ref{conjecture Gross} and $s(k)=r_2(k)$ holds by \cite[Example~4.4.1]{Kataoka2017JNT}. See \cite[Remark~2.2.1]{Kataoka2017JNT} as well. 
\end{proof}
By combining the above results with Corollary~\ref{cor GGC abel imag quad}, we obtain the following. 
\begin{cor}
Suppose the assumptions of Corollary~\ref{cor GGC abel imag  quad}. Then the set
\[
\{k_\infty\in\mathcal{E}_{\mathrm{ns}}(k)\mid \mu(k_\infty/k)=0, \lambda(k_\infty/k)=[k:\Q]/2\}
\]
contains an open dense subset of $\mathcal{E}_{\mathrm{ns}}(k)$ and, in particular, there exists a $\Z_p$-extension $k_\infty/k$ such that $\mu(k_\infty/k)=0$ and $\lambda(k_\infty/k)=[k:\Q]/2$. 
\end{cor}
In the setting of Theorem~\ref{thm Fujii GGC},  if we suppose further that $k/\Q$ is abelian, we can deduce a similar result (see \cite[Theorem~1.3]{Kataoka2017JNT} as well). 

For a fixed prime number $p$ and an integer $\ell\geq 0$, the existence of a $\Z_p$-extension $k_\infty/k$ such that $\lambda(k_\infty/k)=\ell$ is not known in general. This problem is studied in \cite{FOO2006} and a partial result is obtained there (the cases where $p$ is $2$, $3$ or $5$ are solved affirmatively \cite[Theorem~1]{FOO2006}). By the above observation, to solve the problem, it suffices to construct a finite extension $k/\Q$ that satisfies the following. 
\begin{enumerate}
\item $p$ splits completely in $k$. 

\item $k$ is an abelian extension over $\Q$ or an imaginary quadratic field $K$. 

\item  $[k:\Q]=2\ell$ holds. 

\item Conjecture~\ref{conjecture GGC} is true for $k$ and $p$. 
\end{enumerate}
It is an interesting attempt to find such $k$ by using Theorem~\ref{thm Fujii GGC} or Corollary~\ref{cor GGC abel imag quad}.

\subsection{Examples of Greenberg's Generalized Conjecture}\label{subsection example GGC}

\subsubsection{On the non-triviality of the unramified Iwasawa module}
In this section, we show some concrete examples of Theorem~\ref{thm main GGC} such that $X(\widetilde{k})$ is not trivial. We use the following lemmas to check the non-triviality of $X(\widetilde{k})$. 
\begin{lem}[{\cite[Lemma~3.9]{Fujii2022}}]\label{lem imaginary non-trivial}
Let $k/\Q$ be a totally imaginary finite extension and $p$ a prime number that splits completely in $k$. Suppose that $p$-adic Leopoldt's Conjecture is true for $k$. If $[k:\Q]\geq 8$ holds, then $X(\widetilde{k})$ is not trivial. 
\end{lem}
\begin{lem}\label{lem single complex non-trivial}
Let $k/\Q$ be a finite extension, $p$ a prime number that splits completely in $k$. Suppose that $p$-adic Leopoldt's Conjecture is true for $k$. If there exists a $\Z_p$-extension $k_\infty/k$ that is ramified at every $p$-adic prime and satisfies $\rank_p X(k_\infty)\geq r_2(k)+1$, then $X(\widetilde{k})$ is not trivial. 
\end{lem}
\begin{proof}
By \cite[Proposition~1]{Ozaki1997}, we have $\widetilde{k}\subseteq L(k_\infty)$. Hence there is an exact sequence
\[
\begin{tikzpicture}
\node (a) at (0,0) {$0$}; 
\node (b) at (2.5,0) {$\Gal(L(k_\infty)/\widetilde{k})$}; 
\node (c) at (6,0) {$X(k_\infty)$}; 
\node (d) at (9,0) {$\Gal(\widetilde{k}/k_\infty)$}; 
\node (e) at (11,0) {$0$}; 
\draw[->] (a) -- (b); \draw[->] (b) -- (c); \draw[->] (c) -- (d); \draw[->] (d) -- (e); 
\end{tikzpicture}
\]
of $\Z_p$-modules. Since $\rank_pX(k_\infty)\geq r_2(k)+1$ and $\Gal(\widetilde{k}/k_\infty)\simeq\Z_p^{r_2(k)}$ hold, the map $X(k_\infty)\to \Gal(\widetilde{k}/k_\infty)$ is not injective, so $\Gal(L(k_\infty)/\widetilde{k})$ is not trivial. Hence $X(\widetilde{k})$ is not trivial as well, because there is a natural surjection $X(\widetilde{k})\to \Gal(L(k_\infty)/\widetilde{k})$. 
\end{proof}
Let $k/\Q$ be a finite extension, $p$ a prime number, and $k_\infty/k$ a $\Z_p$-extension that is totally ramified at a $p$-adic prime. Then there is a natural surjection $X(k_\infty)\to X(k_n)$ for every $n\geq 0$. So we obtain $\rank_pX(k_\infty)\geq \rank_pX(k_n)$. By using this observation, we may check the inequality in Lemma~\ref{lem single complex non-trivial}. In addition, by using Fukuda's theorem \cite[Theorem~1~(2)]{Fukuda1994}, we may determine the $p$-rank of $X(k_\infty)$. For example, if the $p$-ranks of $X(k_1)$ and $X(k_2)$ coincide, then the $p$-rank of $X(k_\infty)$ also coincides with them. 

For each $n\geq 1$, we let $C_n$ denote the cyclic group of order $n$, $D_n$ the dihedral group of order $2n$, $S_n$ the symmetric group of degree $n$, and $A_n$ the alternating group of degree $n$. All number fields $k$ listed below can be found in LMFDB database. The following examples are checked by using PARI/GP, under the assumption of Generalized Riemann Hypothesis.

\subsubsection{Abelian extensions over imaginary quadratic fields}
Here we pick up some examples that satisfy the assumptions of Corollary~\ref{cor GGC abel imag quad}. All number fields $k$ listed below are not CM-fields so Theorem~\ref{thm  Fujii GGC} does not apply to them. See \cite[Theorem~2.7]{Fujii2022} for an infinite family of imaginary abelian fields $k$ that satisfy the assumptions of Theorem~\ref{thm Fujii GGC} and $X(\widetilde{k})$ is not trivial. It is noteworthy that every imaginary abelian field constructed in \cite{Fujii2022} contains an imaginary quadratic field whose class number is not divisible by the prime number $p$ under consideration. 
\begin{example}
Let $k\coloneqq\Q(\sqrt{7-54\sqrt{-2}})$ and $K\coloneqq\Q(\sqrt{-2})$. Then $k/\Q$ is not a Galois extension, the Galois group of the Galois closure of $k/\Q$ is isomorphic to $D_4$, and $k/K$ is an abelian extension such that $\Gal(k/K)\simeq C_2$. It can be checked that all conditions in Corollary~\ref{cor GGC abel imag quad} are satisfied for $p=3$ and a $3$-adic prime $\Pe$ of $K$. Therefore, Conjecture~\ref{conjecture GGC} is true for $k$ and $p=3$. It can also be checked that the $3$-ranks of $X(k^\cy_1)$ and $X(k^\cy_2)$ are $3$, so the $3$-rank of $X(k^\cy_\infty)$ is $3$. Hence $X(\widetilde{k})$ is not trivial by Lemma~\ref{lem single complex non-trivial}. 
\end{example}
\begin{rem}
Takahashi \cite[Theorem~1.2]{Takahashi2021} showed Conjecture~\ref{conjecture GGC} for quadratic extensions over imaginary quadratic fields under some assumptions. However, Takahashi's result and Corollary~\ref{cor GGC abel imag quad} (specialized to quartic fields) do not cover each other, since it is assumed in \cite[Theorem~1.2]{Takahashi2021} that at least one $p$-adic prime of $K$ does not split in $k$ (in our notation). 
\end{rem}
\begin{example}
Let $k$ be the field defined by $x^6 + 100x^4 + 3085x^2 + 58482$ over $\Q$ and $K\coloneqq\Q(\sqrt{-2})$. Then $k/\Q$ is a Galois extension such that $\Gal(k/\Q)\simeq S_3$ and $k/K$ is an abelian extension such that $\Gal(k/K)\simeq C_3$. It can be checked that all conditions in Corollary~\ref{cor GGC abel imag quad} are satisfied for $p=3$ and a $3$-adic prime $\Pe$ of $K$. Therefore, Conjecture~\ref{conjecture GGC} is true for $k$ and $p=3$. It can be checked that the $3$-rank of $X(k^\cy_1)$ is $3$, but it is not enough to show the non-triviality of $X(\widetilde{k})$. Although it is too heavy to compute the $3$-rank of $X(k^\cy_2)$, we can check the non-triviality of $X(\widetilde{k})$ as follows. Let $K^\an_\infty/K$ be the anti-cyclotomic $\Z_3$-extension over $K$. It can be checked that the first layer $K^\an_1$ is defined by $x^6 - 6x^4 + 9x^2 + 8$ over $\Q$. Now we consider a $\Z_3$-extension $K^\an_\infty k/k$ instead of $k^\cy_\infty/k$. Since $3$ does not divide the class number of $K$, the extension $K^\an_\infty/K$ is totally ramified at every $3$-adic prime. So $K^\an_\infty k/k$ is also totally ramified at every $3$-adic prime because $3$ splits completely in $k$. The first layer of $K^\an_\infty k/k$ is $K^\an_1 k$ and it can be checked that the $3$-rank of $X(K^\an_1k)$ is $4$. Hence we obtain $\rank_3X(K^\an_\infty k)\geq 4$. Therefore, we conclude that $X(\widetilde{k})$ is not trivial by Lemma~\ref{lem single complex non-trivial}. 
\end{example}
\begin{rem}
Itoh \cite[Theorem~1.1]{Itoh2023} showed Conjecture~\ref{conjecture GGC} for imaginary $S_3$-extensions over $\Q$ under some assumptions. However, Itoh's result and Corollary~\ref{cor GGC abel imag quad} (specialized to $S_3$-extensions) do not cover each other, since it is assumed in \cite[Theorem~1.1]{Itoh2023} that $p$ does not split in $K$ (in our notation). 
\end{rem}
\begin{example}
Let $k$ be the field defined by $x^8 - 6x^6 + 39x^4 + 66x^2 + 25$ over $\Q$ and put $K\coloneqq \Q(\sqrt{-1})$. Then $k/\Q$ is not a Galois extension, the Galois group of the  Galois closure of $k/\Q$ is isomorphic to $C_2^2\wr C_2$, and $k/K$ is an abelian extension such that $\Gal(k/K)\simeq C_2\times C_2$. It can be checked that all conditions in Corollary~\ref{cor GGC abel imag quad} are satisfied for $p=5$ and a $5$-adic prime $\Pe$ of $K$. Therefore, Conjecture~\ref{conjecture GGC} is true for $k$ and $p=5$. We also see that $X(\widetilde{k})$ is not trivial by Lemma~\ref{lem imaginary non-trivial}. 
\end{example}

\subsubsection{Fields with a single complex place}
Here we pick up two examples that satisfy the assumptions of Corollary~\ref{cor GGC single complex}. See \cite[\S4]{Kataoka2017} for examples of complex cubic fields. 
\begin{example}
Let $k$ be the field defined by $x^4 - 2x^3 - 23x^2 - 30x - 54$ over $\Q$. Then $k$ is a quartic field that has a single complex place. The Galois group of the Galois closure of $k/\Q$ is isomorphic to $S_4$. It can be checked that all conditions of Corollary~\ref{cor GGC single complex} are satisfied for $p=3$ and a $3$-adic prime $\pe_0$ of $k$. It can also be checked that the $3$-ranks of $X(k^\cy_1)$ and $X(k^\cy_2)$ are $2$, so the $3$-rank of $X(k^\cy_\infty)$ is $2$. Hence $X(\widetilde{k})$ is not trivial by Lemma~\ref{lem single complex non-trivial}. 
\end{example}
\begin{example}
Let $k$ be the field defined by $x^6 - 16x^4 + 29x^2 + 311$ over $\Q$. Then $k$ is a sextic field that has a single complex place. The Galois group of the Galois closure of $k/\Q$ is isomorphic to $A_4\times C_2$. It can be checked that all conditions of Corollary~\ref{cor GGC single complex} are satisfied for $p=5$ and a $5$-adic prime $\pe_0$ of $k$. It can also be checked that the $5$-rank of $X(k^\cy_1)$ is $2$, so we obtain $\rank_5X(k^\cy_\infty)\geq 2$. Hence $X(\widetilde{k})$ is not trivial by Lemma~\ref{lem single complex non-trivial}. 
\end{example}

\subsubsection{Other types of number fields}
As we noted in Remark~\ref{rem conditions}, Theorem~\ref{thm main GGC} conjecturally applies to arbitrary types of number fields. Here we pick up two examples that are not contained in Theorem~\ref{thm Fujii GGC}, Corollaries~\ref{cor GGC abel imag quad}, and \ref{cor GGC single complex}. 
\begin{example}
Let $k$ be a field defined by $x^5-x^4+5x^3+x^2-15x-18$ over $\Q$. Then $k$ is a quintic field with $(r_1(k), r_2(k))=(1,2)$. The Galois group of the Galois closure of $k/\Q$ is isomorphic to $S_5$. It can be checked that all conditions of Theorem~\ref{thm main GGC} are satisfied for $p=3$ and a set $\Sigma$ of $3$-adic primes of $k$. It can also be checked that the $3$-ranks of $X(k^\cy_1)$ and $X(k^\cy_2)$ are $3$, so the $3$-rank of $X(k^\cy_\infty)$ is $3$. Hence $X(\widetilde{k})$ is not trivial by Lemma~\ref{lem single complex non-trivial}. 
\end{example}
\begin{example}
Let $k$ be a field defined by $x^6-3x^5+7x^4+9x^3+19x^2+21x-27$ over $\Q$. Then $k$ is a sextic field with $(r_1(k), r_2(k))=(2,2)$. The Galois group of the Galois closure of $k/\Q$ is isomorphic to $S_4\times C_2$. It can be checked that all conditions of Theorem~\ref{thm main GGC} are satisfied for $p=3$ and a set $\Sigma$ of $3$-adic primes of $k$. It can also be checked that the $3$-rank of $X(k^\cy_1)$ is $3$, so we obtain $\rank_3X(k^\cy_\infty)\geq 3$. Hence $X(\widetilde{k})$ is not trivial by Lemma~\ref{lem single complex non-trivial}. 
\end{example}

\section*{Use of AI}
The author used ChatGPT (GPT-6, OpenAI) for mathematical discussions in \S\ref{subsection splits completely} and \S\ref{subsection conditional} (both added in \texttt{v2} of the manuscript).  A related remark is given in Remark~\ref{rem conditions}. ChatGPT was also used to check the manuscript for grammatical and typographical errors.

\section*{Acknowledgements}
I am grateful to my supervisor Takenori Kataoka for carefully reading the manuscript and providing valuable comments.

\printbibliography
\end{document}